%% file: dilog_from_pentagon_v4.0.tex
\documentclass{amsart}
\usepackage{hyperref}
\usepackage{amsmath,amssymb}
\usepackage{graphicx}
\usepackage{xcolor}
\usepackage{bm}
\usepackage{txfonts}
\usepackage{url}
\usepackage{breakurl} 
\usepackage{breqn}
\usepackage{anyfontsize}

\newtheorem{Thm}{Theorem}
\newtheorem{Lem}{Lemma}
\newtheorem{Prop}{Proposition}
\newtheorem{Cor}{Corollary}
\newtheorem{Rem}{Remark}
\newtheorem{Def}{Definition}
\newtheorem{Con}{Conjecture}
\newtheorem{Ex}{Example}
\newtheorem*{Property}{Property}

\def\inv{^{-1}}

\allowdisplaybreaks

\begin{document}

\title{Comments on Exchange Graphs in Cluster Algebras}

\author{Hyun Kyu Kim}
\address{(HK) Department of Mathematics, Ewha Womans University, 52 Ewhayeodae-gil, Seodaemun-gu, Seoul 03760, Republic of Korea}
\email{hyunkyukim@ewha.ac.kr}

\author{Masahito Yamazaki}
\address{(MY) Kavli IPMU (WPI), University of Tokyo, Chiba 277-8583, Japan; and Center for Fundamental Laws of Nature, Harvard Unversity, Cambridge MA 02138, USA}
\email{masahito.yamazaki@ipmu.jp}



\date{\today}

\begin{abstract}
An important problem in the theory of cluster algebras is to compute the fundamental group of the exchange graph. A non-trivial closed loop in the exchange graph, for example, generates a non-trivial identity for the classical and quantum dilogarithm functions. An interesting conjecture, partly motivated by dilogarithm functions, is that this fundamental group is generated by closed loops of mutations involving only two of the cluster variables. We present examples and counterexamples for this naive conjecture, and then formulate a better version of the conjecture for acyclic seeds. 
\end{abstract}

\maketitle
\tableofcontents

\section{Introduction and Summary}\label{sec.intro}

Cluster algebra is a mathematical framework introduced by 
Fomin and Zelevinsky \cite{FZ1,FZ2,FZ4}.
Cluster algebra is defined from a skew-symmetrizable $n\times n$ integer matrix $B$
and an $n$-tuple of coefficients in a fixed semi-field.
We start with an initial seed, and we generate other seeds by repeating 
a combinatorial procedure known as a mutation.

An \emph{exchange graph} $\Gamma$ for a cluster algebra is a graph whose vertices are labeled by seeds, and whose edges by mutations (an edge connects two vertices 
if the corresponding two seeds are related by the mutation associated to the edge).
The exchange graph is  by definition connected.

It is an important problem of theory of the cluster algebras 
to identify the fundamental group of the exchange graph.
For example, we can associate a quantum dilogarithm identity \cite{ReinekePoisson,KellerOn,KashaevN} as well as a classical dilogarithm identity \cite{Nakanishi_Periodicities}, to any non-trivial closed loop of the exchange graph. 
The fundamental group can also be thought of as defining relations for the so-called cluster groupoid.

In this short note, we study the fundamental group of the exchange graph.
We first formulate Property ($\star$) in section \ref{sec.property}.
This property holds for finite-type seeds (section \ref{sec.finite}),
whereas there are counterexamples for non-finite-type cases (section \ref{sec.non-finite}).
In section \ref{sec.conjecture} we formulate our conjecture for acyclic cases 
(Conjecture \ref{conj1}). We will also discuss implications to quantum dilogarithm functions in section \ref{sec.dilog}.

While many of the results presented here are strictly speaking not new, 
we here tried to provide a coherent presentation incorporating the known results/examples/counterexamples scattered in the literature, and propose a concrete well-defined conjecture.
We hope that our small contribution will facilitate further developments in this exciting subject.

We would like thank Kyungyong Lee for discussion, and the referee(s) for reviewing and suggestions.
M.Y.\ was supported in part by WPI Research Center
Initiative (MEXT, Japan), by JSPS Program for Advancing Strategic
International Networks to Accelerate the Circulation of Talented Researchers,
by JSPS KAKENHI Grant Number 15K17634, and by JSPS-NRF Joint Research Project.

For H.K.: This work was supported by the Ewha Womans University Research Grant of 2017. This research was supported by Basic Science Research Program through the National Research Foundation of Korea(NRF) funded by the Ministry of Education(grant number 2017R1D1A1B03030230).


\section{Property \texorpdfstring{($\star$)}{(*)}}\label{sec.property}

Let us begin by asking if the following property holds:
\begin{Property}[$\star$]
The fundamental group for the exchange graph is generated by 
elements of the form $\mathcal{P} \mathcal{L} \mathcal{P}^{-1}$, where $\mathcal{L}$ is a closed loop (with a base point) obtained by 
mutating $2$ of the $n$ cluster variables while keeping the remaining $n-2$ variables fixed,
$\mathcal{P}$ is an arbitrary path originating at the base point, and $\mathcal{P}^{-1}$ is its inverse.
\end{Property}

\begin{Rem}\label{Rem.finite}
The exchange graph is in general an infinite graph. In the definition of the fundamental group,
we allow only finite closed loops.
\end{Rem}

Let us explain the motivation for this property, originating from the classical dilogarithm identity.

As already commented before, for a closed loop of the exchange graph we can associate an identity for the classical (Rogers) dilogarithm function $L(x)$.\footnote{
This is defined by
\begin{align}
\begin{split}
L(x):=
-\frac{1}{2}\int_0^x \! dt\, 
\left(
 \frac{\log(1-t)}{t} 
+ \frac{\log t}{1-t}
 \right) \ .
\end{split} 
\label{LRogers}
\end{align}
}
In the notations of \cite{KashaevN,Ip:2014pva},
this identity takes the form
\begin{align}
\sum_{t=1}^M \epsilon_t \,  L\left( \frac{y_{k_t}(t)^{\epsilon_t} }{1+y_{k_t}(t)^{\epsilon_t}} \right) = 0  \;.
\label{L_sum}
\end{align}
Here we consider $M$ mutations at $(k_1, \ldots, k_M)$, with the (classical) $y$-variables 
after $t$ mutations denoted by $y_i(t)$ (with $i=1, \ldots, n$), while $\epsilon_j=\pm 1$ are the so-called tropical signs.
For our purposes it is important that the variables $y_{k_t}(t)$, and hence the arguments $(y_{k_t}(t)^{\epsilon_t})/(1+y_{k_t}(t)^{\epsilon_t})$ of the classical dilogarithm function, 
are rational functions of the $n$ cluster variables $\{ y_i(0) \}_{i=1}^n$ in the initial seed.

Now, one characteristic identity satisfied by the classical dilogarithm function $L(x)$ is the celebrated five-term identity
\begin{align}
L(x)+L(y)=L\left(
\frac{x(1-y)}{1-xy}
\right)
+
L\left(
xy
\right)
+
L\left(
\frac{y(1-x)}{1-xy}
\right) \;,
\label{Lpentagon}
\end{align}
which is known to be the classical dilogarithm identity \eqref{L_sum}
for the $A_2$ matrix involving five mutations.

If all the functional identities (including the classical dilogarithm identity of \eqref{L_sum}) follow from repeated use of the pentagon \eqref{Lpentagon}, that gives a strong indication (albeit not a proof) that Property ($\star$) holds.\footnote{We can formulate this problem more intrinsically at the level of the Bloch group.
Suppose that we have a set of rational functions $x_i(\vec{t})$ with respect to variables $\vec{t}$ satisfying $\sum_{i} c_i x_i(\vec{t}) \wedge (1-x_i(\vec{t})) =0$. Then $\sum_i c_i [x_i(\vec{t})]$ defines an element of the Bloch group, and the question is if this element is trivial in the Bloch group.}
In this connection, it is worth mentioning the result of Wojtkowiak:
\begin{Thm}[\cite{Wojtkowiak}]
\label{thm.Woj}
Any functional equation of the dilogarithm with rational functions of one variable as arguments is a consequence of the five-term relation (up to a constant).
\end{Thm}
One might be tempted to regard this theorem as a supporting evidence for Property $(\star)$.
Unfortunately, it is known that Theorem \ref{thm.Woj} does not generalize to the case of multiple variables as arguments ({\it cf.} \cite{Zagier_Frontiers}).\footnote{Five-term identity ``almost'' determines the function $L(x)$; a one-variable function satifying the pentagon \eqref{Lpentagon} as well as  the inversion relation
$L(x)+L(1-x)=\pi^2/6$ and differentiable three times or more coincides with $L(x)$ \cite[section 4]{RogersOld}).
This result in itself, however, does not guarantee that \eqref{L_sum} arises from repeated use of the five-term identity \eqref{Lpentagon}.
} Whether or not Property ($\star$) holds or not therefore should reflect this subtlety, to say the least.

\begin{Rem}
The reference \cite{Terashima:2013fg} associates a three-dimensional $\mathcal{N}=2$ supersymmetric gauge theory 
for a sequence of quiver mutations, for a skew-symmetric matrix $B$. We can then 
also associate a non-trivial duality (more precisely an equivalence of the $S^3$ partition function) between a pair of three-dimensional $\mathcal{N}=2$ supersymmetric gauge theories \cite{Terashima:2013fg,Gang:2015wya}.
Property ($\star$) in this context means that any such duality can be obtained starting with the 
three-dimensional $\mathcal{N}=2$ mirror symmetry between supersymmetric quantum electro-dynamics (SQED) and the XYZ model \cite{Aharony:1997bx}, together with gauging of global symmetries and adding superpotential interactions (with monopole operators in general) \cite{Dimofte:2011py}.
In other words, if Property ($\star$) does not hold then it is an indication that there is a new duality between three-dimensional $\mathcal{N}=2$ theories unknown in the literature. For this reason Property ($\star$) is also of great interest to
physicists.
\end{Rem}

\section{Finite Type Case}\label{sec.finite}

Let us discuss the case of finite type seeds, where the exchange graph is a finite graph.
Such a seed is classified by Fomin and Zelevinsky \cite{FZ2}. 

\begin{Thm}[\cite{FZ2}]
\label{thm:finite_type}
Property ($\star$)  holds for finite type seeds.
\end{Thm}

\begin{proof}

The proof below is essentially contained in \cite{FZ2}.
However, the significance of their result in the current context of Property ($\star$) is not 
emphasized too explicitly in the paper, and hence we find it worth spending a few paragraphs on the proof.

The exchange graph is dual to the cluster complex $\Delta$, 
which is a $(n-1)$-dimensional simplicial complex whose ground set
is the set of all cluster variables and whose maximal simplices are the clusters. 
Namely, a $(d-1)$-dimensional simplex is given by a $d$-element subset of a single cluster.
In particular, the top-dimensional simplex, namely an $(n-1)$-dimensional simplex, corresponds to a cluster.

The cluster complex $\Delta$ coincides with another simplicial complex 
$\Delta(\Phi)$ defined from the root system $\Phi$ for the associated Dynkin type  
(\cite{FZ2}, in particular Theorem 1.13).
Now Theorem 1.4 of \cite{CFZ} (Theorem 3.2 of \cite{FZ2}) says that the (complete) simplicial fan consisting of the cones spanned by simplices of $\Delta(\Phi)$ is the normal fan of a simple $n$-dimensional convex polytope. Lemma 2.2 of \cite{FZ2} then implies 
that the fundamental group of the exchange graph is generated by 
elements of the form $\mathcal{P} \mathcal{L} \mathcal{P}^{-1}$, where $\mathcal{L}$ is a `geodesic loop',
$\mathcal{P}$ is an arbitrary path originating at the base point, and $\mathcal{P}^{-1}$ is its inverse.

It therefore remains to identify geodesic loops with mutations involving only two cluster variables.
In \cite{FZ2} one picks up a $(n-2)$-dimensional simplicial complex $D$, and 
define the simplicial complex $\Delta_D$ by a quotient.
Namely, this is a simplicial complex of the quotient cluster algebras,
such that $D'$ is a simplex of $\Delta_D$ if and only if $D\cup D'$ is a simplex in $\Delta$.
In the language of the cluster variables, this means that we fix $n-2$ cluster 
variables, and then consider the mutations on the remaining two variables.
The geodesic loop is identified to be a loop inside this $1$-dimensional complex $\Delta_D$. This completes the identification.

\end{proof}

\begin{Cor}\label{Thm_groupoid}
For finite type seeds, the loop $\mathcal{L}$ in property ($\star$)
corresponds to a sequence of mutations of rank-two skew-symmetrizable matrices of type either $A_1\times A_1, A_2, B_2$ or $G_2$, and of length $4,5,6,8$, respectively.
\end{Cor}

\begin{proof}
For a finite type seed, we have $|b_{i,j} b_{j,i}|=0,1,2,3$ for any $i, j=1, \ldots, n$ \cite{FZ2}.
The $2\times 2$ matrix defined by $(b_{i,j})$ is then of type $A_1\times A_1, A_2, B_2$ or $G_2$, respectively. Since the matrix is of rank $2$, for each case it is then straightforward to repeat mutations to identify the closed loop in the exchange graph, and to see that it has length $4,5,6,8$.
\end{proof}

\section{Non-Finite-Type Cases}\label{sec.non-finite}

Let us next move onto non-finite-type cases. For this purpose, it is useful to quote the 
results for cases where the seed is generated from 
a signed adjacency matrix of an ideal triangulation of a Riemann surface with marked points, as described in \cite{FST}.

\begin{Thm}[\cite{FST},\cite{FT}]
Property ($\star$)
holds for seeds generated from a signed adjacency matrix of an ideal triangulation of a bordered surface with marked points, except when the surface is a closed surface with exactly two punctures.
\end{Thm}

Let us define a graph $\Gamma'$ to be a graph whose vertex is an ideal triangulation
and whose edge connecting two vertices is an operation called a flip, mapping one ideal triangulation to another. The following theorem guarantees that the counterpart of Property ($\star$) holds for this graph:

\begin{Thm}[\cite{FST}, Theorem 3.10\footnote{This theorem is known long before \cite{FST}, 
see e.g.\ \cite{Harer} and references in \cite{FST} for more detailed literature list.}]
The fundamental group of  $\Gamma'$ is  
generated by cycles of length $4$ corresponding to pairs of commuting flips,
and the cycles of length $5$ whose removal would create a pentagonal face.
\label{Thm.ideal}
\end{Thm}

There is one cautionary remark, however. As emphasized in \cite{FST}
the graph $\Gamma'$ is coming from ideal triangulations in general only a subgraph of the 
exchange graph $\Gamma$, and the fundamental group of the former is in general a subgroup of that of the latter. This is because we cannot flip at an edge surrounded by a self-folded triangle,
whereas in the cluster algebras we can mutate at any vertex of the quiver diagram.

For this reason, Theorem \ref{Thm.ideal} does not imply Property ($\star$).
In fact, we can find an example of a surface-type seed which does \emph{not}
satisfy Property ($\star$):


\begin{Prop}[\cite{FST}, Remark 9.19; \cite{FT}] \label{Thm.counter}
Property ($\star$) does not hold for a seed associated to an ideal triangulation of a closed surface 
(of genus $1$ or higher) with two punctures.
\end{Prop}



\begin{Ex}\label{ex.genus1}

Let us give an example for the genus $1$ surface with two punctures.
In this case, we can start with the ideal triangulation of Fig.~\ref{fig.g1quiver},
whose associated exchange matrix is given by
\begin{align}
B=
\left(
\begin{array}{cccccc}
0 & 0 & -1 & 1 & -1 & 1 \\ 
0 & 0 & 1 & -1 & 1 & -1 \\ 
1 & -1 & 0 & 1 & 0 & -1 \\ 
-1 & 1 & -1 & 0 & 1 & 0 \\ 
1 & -1 & 0 & -1 & 0 & 1 \\ 
-1 & 1 & 1 & 0 & -1 & 0 
\end{array}
\right) \;.
\label{B_genus1}
\end{align}
\begin{figure}[htbp]
\centering{\includegraphics[scale=0.65]{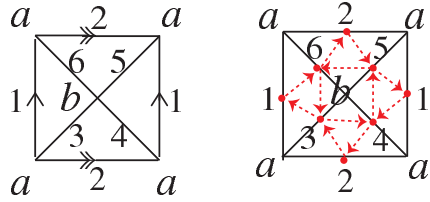}}
\caption{((Left) An ideal triangulation of the torus with two punctures.
The left and right (top and bottom) edges are to be identified.
The edges are labeled by $1, 2, \ldots, 6$, and punctures by $a, b$.
(Right) an associated quiver, whose signed adjacency matrix gives the exchange matrix $B$ in \eqref{B_genus1}.}
\label{fig.g1quiver}
\end{figure}

We found a mutation sequence corresponding to a closed loop in the exchange graph,
which is of length $32$:
\begin{align}
5, 6, 4, 3, 6, 5, 1, 2, 4, 3, 6, 5, 3, 4, 2, 1, 6, 5, 3, 4, 5, 6, 1,
2, 3, 4, 5, 6, 4, 3, 2, 1 \;,
\end{align}
For completeness we also have verified numerically that the cluster comes back to itself after this mutation sequence.

\end{Ex}


To better understand this mutation sequence geometrically, it is useful to use the concept of a
tagged triangulation. This concept is introduced in order to allow for a mutation at an edge encircled by a self-folded triangle. In practice, a tagged triangulation is defined to be a maximal collection of pairwise compatible tagged arcs, see \cite[section 7]{FST} for details. In particular, each of the two endpoints of a tagged arc has an extra label, either `plain' or `notched'. Graphically we can represent the `notched' labeling by the symbol $\Join$.

The length $32$ mutation sequence then can be shown as in Figure \ref{fig.g1_loop}.
Notice that the signatures of tags\footnote{We can associate a signature $+1, 0, -1$ to each puncture, and the set of such numbers define strata \cite[section 9]{FST}.} around the two punctures are different for all the eight tagged triangulations 
shown in the figure. This corresponds to the fact that we are going around eight different strata
in the language of  \cite{FST}.

\begin{figure}[htbp]
\centering{\includegraphics[scale=0.65]{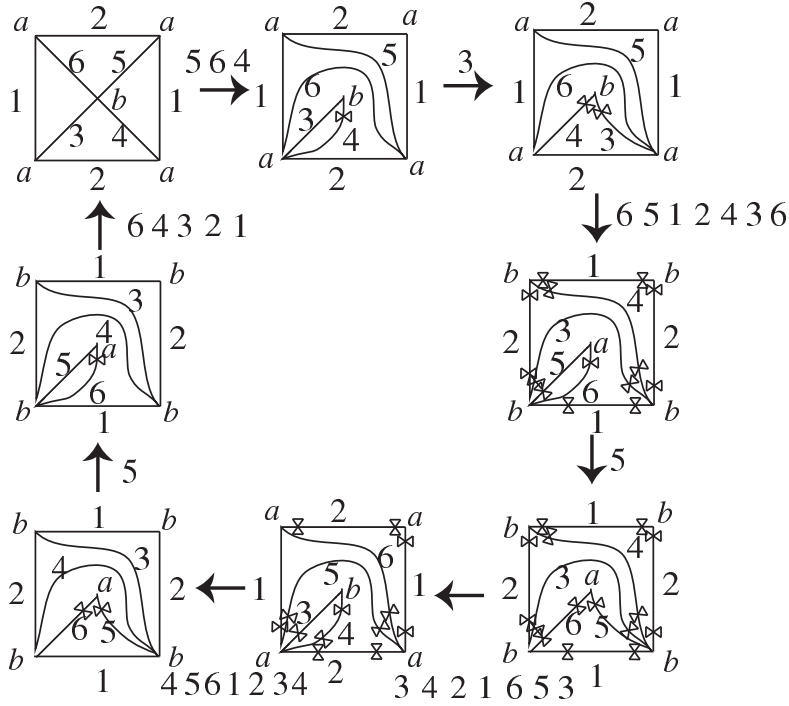}}
\caption{A closed loop in the exchange diagram, for a seed corresponding to a tagged triangulation of a genus $1$ surface with two punctures. The eight tagged triangulations shown here all have different signatures of tags around the puncture, and belong to different strata \cite[Remark 9.19]{FST}. 
}
\label{fig.g1_loop}
\end{figure}


\begin{Ex}\label{ex.genus2}

We can generalize Example \ref{ex.genus1} to a general $g>1$ surface with two punctures.

\begin{figure}[htbp]
\centering{\includegraphics[scale=0.6]{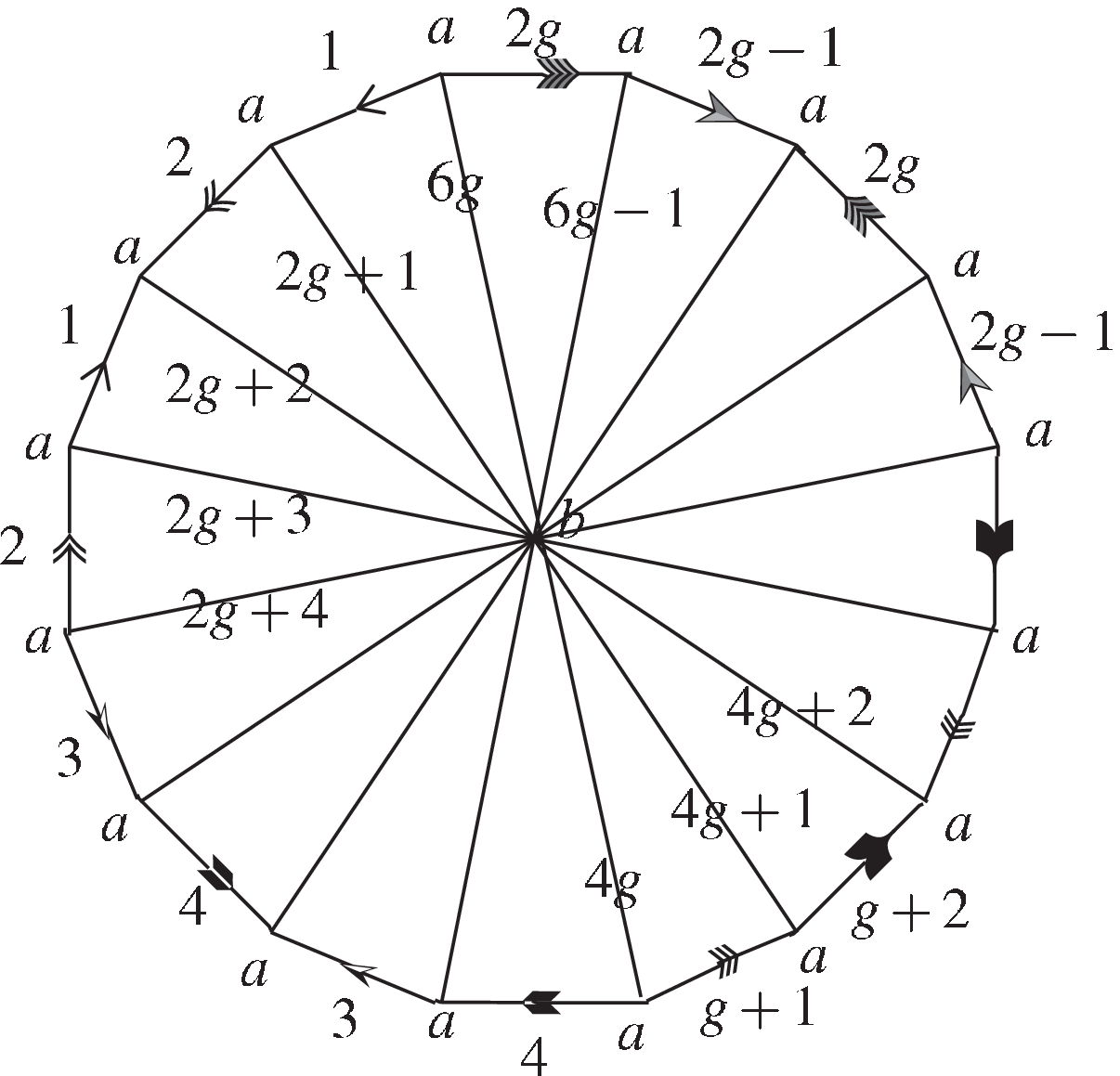}}
\caption{An ideal triangulation of a genus $g$ surface with two punctures, 
shown here is the case of $g=4$.
On the boundary of this $4g$-gon, edges with the same labels are to be identified and glued together.
We have $6g$ edges and two punctures $a, b$.}
\label{fig.g2quiver}
\end{figure}

Let us triangulate the surface as in Fig.~\ref{fig.g2quiver}.
We find the mutation sequence in this case is as follows. First, we consider flip at $4g-2$ edges, while still in the same stratum:
\begin{align}
2g+2, 2g+3, \ldots, 4g-1, 4g, \quad 4g+2, 4g+3, \ldots, 6g \;.
\end{align}
We then flip at edges $2g+1$ and then $4g+1$, each time changing tags at one of the punctures.
Afterwards we flip at edges
\begin{align}
6g, 6g-1, \ldots, 4g+3, 4g+2, \quad 4g, 4g-1,  \ldots, 2g+2,\quad  1,2, 3,4, \ldots, 2g \;.
\end{align}
After all these $10g-2$ flips, 
the triangulation comes back to itself modulo labels of edges and the punctures interchanged.
We can then repeat the similar mutation sequence three more times, and after $4(10g-2)=40g-8$ steps we are back to the initial tagged triangulation. 

\end{Ex}

\section{Conjecture}\label{sec.conjecture}

For those seeds not coming from ideal triangulations of Riemann surfaces with marked points,
there seems to be no general known results applicable in general.

Let us here state our conjecture:

\begin{Con}\label{conj1}
Property ($\star$)  holds for seeds of acyclic type.
\end{Con}

\begin{Rem}
Towards the completion of this paper it came to our attention that the same conjecture was made in 
\cite[Conjecture 8.15]{Warkentin}. Note \cite{FZ1} (above Example 7.8) contains a similar conjecture, however without the restriction to acyclic case (recall the counterexample of Proposition 
\ref{Thm.counter}).
\end{Rem}

An interesting example is provided by the so-called $(G, G')$ quiver, 
specified by a pair of Dynkin diagrams \cite{KellerPairs}.
This quiver has a known mutation sequence which sends the seed back to itself
(this is the statement of the periodicity of the $Y$-system \cite{Zamolodchikov:1991et,Kuniba:1992ev,Gliozzi:1995wq}).
We can try to see if this loop of the exchange graph arises from the pentagon.
For example, for the $(A_3, A_3)$ quiver the loop of the exchange diagram is of length $36$.
With the help of computer program we have verified that this follows from the repeated
use of closed loop corresponding to rank $2$ skew-symmetric matrices of type $A_2$ and $A_1\times A_1$
(see Appendix \ref{app.A3A3} for details).
We have repeated similar computations for some other values of $k$ and $n$, and 
verified that the conjecture holds.
It would be interesting to prove this result for general values of $k$ and $n$.

\section{Quantum Dilogarithm Identities}\label{sec.dilog}

As alluded to before, a loop in the exchange graph gives rise to a quantum dilogarithm identity.
Let us study the implications of these results at the level of the quantum dilogarithm function.

\begin{Def}[Quantum Dilogarithm Function \cite{Faddeev:1993rs,Faddeev:1993pe}]
We define the compact quantum dilogarithm function $\Psi_q(x)$ by\footnote{
In some literature $q^2$ is denoted by $q$.
In our notation we always have integer powers of $q$.
}
\begin{align}
\Psi_q(x):=\frac{1}{(-q x;q^2)_{\infty}} \;,
\quad
(x;q)_{\infty}: = \prod_{k=0}^{\infty} (1-q^k x) \;.
\label{Psi_def}
\end{align}
and the non-compact dilogarithm $\Phi^{\hbar}(z)$ by an integral expression
valid in the region  $|\textrm{Im}(z)| < \pi(1+\hbar)$:
\begin{align}
\begin{split}
\Phi^{\hbar}(z)
&  :=\exp \left( -\frac{1}{4} \int_{\Omega}
 \frac{e^{-i p z} }{\sinh(\pi p) \sinh (\pi \hbar p) } \frac{dp}{p}  \right)  \;,
 \end{split}
 \label{Phi_def}
\end{align} 
where $\hbar$ is a positive real number and the contour $\Omega$ is along the real axis, avoiding the pole at the origin from above along a small half-circle.

\end{Def}


The discussion of quantum dilogarithm identity is parallel between compact and non-compact dilogarithms \cite{KashaevN}, and we therefore mostly state only the compact cases, to avoid repetition.

\begin{Thm}[Quantum Dilogarithm Identity in Tropical Form \cite{ReinekePoisson,KellerOn}]\label{qdilog_tropical}

Consider a closed loop of the exchange graph, and let us denote the 
associated sequence of mutations by the labels $k_1, k_2, \ldots, k_M$.
Let $c_t$ ($t=1\ldots, M$) be the $c$-vector of the classical $y$-variables, and let us
denote their tropical signs by $\epsilon_t$ ($t=1\ldots, M$).
We then have
\begin{align}
\Psi_q(Y_{\epsilon_1 c_1})^{\epsilon_1} 
\cdots 
\Psi_q(Y_{\epsilon_M c_M} )^{\epsilon_M}=1 \;,
\label{eq.qdilog_tropical}
\end{align}
where the quantum $y$-variables $Y_{\alpha}$ with $\alpha=\sum_{i=1}^n \alpha_i e_i \in \mathbb{Z}^n$
satisfy the relation (with $Y_i:=Y_{e_i}$)
\begin{align}
Y_{\alpha+\beta}=q^{\langle \alpha, \beta \rangle} Y_{\alpha} Y_{\beta} \;,
\quad \langle \alpha, \beta \rangle=- \langle \beta, \alpha \rangle: =\alpha^T B \beta \;.
\end{align}

\end{Thm}

These general quantum dilogarithm identities have
applications to a wide-ranging topics in mathematics and physics, e.g.\ references in \cite{KashaevN,Ip:2014pva}.


\begin{Prop}\label{Psi_rank2}
The dilogarithm identity for 
rank $2$ finite-type skew-symmetrizable matrix $B$ are given as follows:

For $A_1\times A_1$ with $Y_1Y_2 = Y_2 Y_1$,
\begin{align}
\Psi_q(1)\, \Psi_q(2) =\Psi_q(1)\, \Psi_q(1) \;.
\label{Psi_commute}
\end{align}

For $A_2$ with $Y_1Y_2 = q^2 Y_2 Y_1$:
\begin{align}
\Psi_q(1)\, \Psi_q(2) =\Psi_q(2)\, \Psi_q(12) \, \Psi_q(1) \;.
\end{align}

For $B_2$ with $Y_1Y_2 = q^4 Y_2 Y_1$:
\begin{align}
\Psi_q(1)\, \Psi_{q^2}(2) =\Psi_{q^2}(2)\, \Psi_q(12) \, \Psi_{q^2}(1^2 2) \, \Psi_q(1) \;.
\label{Psi_B2}
\end{align}

For $G_2$ with $Y_1Y_2 = q^6 Y_2 Y_1$:
\begin{align}
\Psi_{q^3}(1) \Psi_q(2) =
\Psi_{q}(2) \Psi_{q^3}(12^3) \Psi_q(1 2^2) \Psi_{q^2}(1^2 2^3) \Psi_q(12) \Psi_{q^3}(1) \;.
\label{Psi_G2}
\end{align}

\end{Prop}

In the expressions above we used the shorthand notation 
$\Psi_q(1)=\Psi_q(Y_1)$, $\Psi_q(12)=\Psi_q(Y_{e_1+e_2})=\Psi_q(q\inv Y_1Y_2)$, $\Psi_q(1^2 2)=\Psi_q(Y_{2e_1+e_2})$, etc. For example,
the $A_2$ identity is the pentagon identity
\begin{align}
\Psi_q(Y_1)\, \Psi_q(Y_2) =\Psi_q(Y_2)\, \Psi_q(q\inv Y_1Y_2) \, \Psi_q(Y_1) \;.
\label{Psi_pentagon}
\end{align}


\begin{proof}
This is by explicit computation, similar to the $D_4$ case in Appendix \ref{app.D4}.
\end{proof}


\begin{Prop}\label{prop.tmp}
Let $B$ be a skew-symmetrizable matrix such that its seed is of finite type. 
Then the associated quantum dilogarithm identities are generated by
rank $2$ identities in Proposition
\ref{Psi_rank2}.
\end{Prop}


\begin{proof}
This follows from Theorem \ref{thm:finite_type}.
\end{proof}


\begin{Lem}[Folding Formulas for Quantum Dilogarithm]
We have for the compact dilogarithm
\begin{align}
\begin{split}
\Psi_q(x)&=\Psi_{q^2}(qx) \,\Psi_{q^2}(q^{-1}x)
 \\
&=\Psi_{q^3}(q^2x)\, \Psi_{q^3}(x) \,
 \Psi_{q^3}(q^{-2}x) 
 \\
 &=\prod_{j=1}^k \Psi_{q^k}(q^{k-2j+1} x) \;.
\end{split}
\label{Psi_fold}
\end{align}
and for the non-compact quantum dilogarithm
\begin{align}
\begin{split}
\Phi^{\hbar}(z)&=\Phi^{2\hbar}(z+i\pi \hbar)\, \Phi^{2\hbar}(z-i\pi \hbar) \\
&=
\Phi^{3\hbar}(z+2i\pi \hbar)\, \Phi^{3\hbar}(z)\, \Phi^{3\hbar}(z-2i\pi \hbar) \\
&=\prod_{j=1}^k \Phi^{k\hbar}(z+(k-2j+1)i \pi \hbar) \;.
\end{split}
\label{Phi.fold}
\end{align}

\end{Lem}


\begin{proof}

These identities can easily be verified from definitions of the quantum dilogarithm functions (recall \eqref{Psi_def} and \eqref{Phi_def}).
For example, from  \eqref{Phi_def}
we find (in the region  $|\textrm{Im}(z)| < \pi(1+\hbar)$)
\begin{align}
\begin{split}
\prod_{j=1}^k \Phi^{k\hbar}\left(z+(k-2j+1)\hbar\right)
&  =\exp \left( -\frac{1}{4} \int_{\Omega}
 \frac{\sum_{i=1}^k e^{-i p (z+i\pi( k-2i+1)\hbar)}}{\sinh(\pi p) \sinh (k \pi \hbar p) } \frac{dp}{p}  \right) \\  &  =\exp \left( -\frac{1}{4} \int_{\Omega}
 \frac{e^{-i p z} \,  \sin \left(k\pi p \hbar\right) / \sin \left(\pi p \hbar\right)}{\sinh(\pi p) \sinh (k\pi \hbar p) } \frac{dp}{p}  \right) \\
   &  =\exp \left( -\frac{1}{4} \int_{\Omega}
 \frac{e^{-i p z} }{\sinh(\pi p) \sinh \left(\pi \hbar p\right) } \frac{dp}{p}  \right) \\
 &= \Phi^{\hbar}(z) \;.
 \end{split}
\end{align}
The identity can then be analytically continued to the whole complex plane.
\end{proof}


\begin{Lem}\label{Lem_B2G2}
The $B_2$ and $G_2$ identities follow from the $A_2$ identity as well as
\eqref{Psi_fold}.
\end{Lem}


\begin{proof}
Let us work out the $G_2$ quantum dilogarithm identity, since the $B_2$ case is similar (and easier)\footnote{
During the preparation of this work we came to aware that 
the fact that $G_2$ identity follows from the pentagon identity is known to some experts, including Gen Kuroki \cite{Kuroki}.
We would like to thank him for correspondence.
}.
The $G_2$ identity \eqref{Psi_G2} follows from the $D_4$ identity
\begin{align}
\begin{split}
&\Psi_q(4) \Psi_q(1) \Psi_q(2) \Psi_q(3) =
\Psi_q(1) \Psi_q(2) \Psi_q(3) \Psi_q(1234)
\\
&\quad\times \Psi_q(234) \Psi_q(134) \Psi_q(124) \Psi_q(1234^2) \Psi_q(14)
\Psi_q(24) \Psi_q(34) \Psi_q(4)
\end{split}
\label{D4_qdilog}
\end{align}
by the substitution $1,2,3\to 1$, $4\to 2$ (this is a version of the standard folding trick for non-simply-laced Dynkin diagram).
The equation \eqref{D4_qdilog} in itself can easily be proven 
by repeated application of the pentagon relation, or from the general quantum dilogarithm identity \eqref{eq.qdilog_tropical}
(see appendix \ref{app.D4}).
\end{proof}


\begin{Thm}\label{Thm_qdilog}
Let $B$ be skew-symmetrizable matrix such that its cluster is of finite type. 
Then the associated quantum dilogarithm identities are generated by the following
three identities, and their conjugation by some operators:
$A_1\times A_1$ identity \eqref{Psi_commute},
$A_2$ identity (pentagon) \eqref{Psi_pentagon},
and the folding formula \eqref{Psi_fold} with $k=2,3$.

\begin{proof}

This follows from Lemma \ref{Lem_B2G2} and Proposition \ref{prop.tmp}.
\end{proof}
\end{Thm}


\begin{Rem}
Since we can prove \eqref{Psi_pentagon}, \eqref{Psi_commute} and \eqref{Psi_fold} directly from the definition of 
the quantum dilogarithm, Theorem \ref{Thm_qdilog} gives a direct proof of
quantum dilogarithm identities for finite type quivers.
\end{Rem} 

\begin{Rem}

By the same arguments
Theorem \ref{Thm_qdilog} can be immediately generalized to several other versions of 
quantum dilogarithm identities. This 
includes the quantum dilogarithm identities in the so-called ``universal forms'' of \cite{KashaevN},
as well as those for the cyclic dilogarithm \cite{Ip:2014pva}.
\end{Rem}

\begin{Rem}
For seeds originating from four-dimensional $\mathcal{N}=2$ gauge theories (such as supersymmetric Yang-Mills theory
with and without matter), the Kontsevich-Soibelman wall crossing formula \cite{Kontsevich:2008fj} for the BPS spectrum 
gives rise to an an identity involving an infinite number of quantum dilogarithm (e.g.\ \cite{Dimofte:2009tm}). However, we have excluded those infinite closed loops from our definition of a fundamental group, see Remark \ref{Rem.finite}.
\end{Rem}

\appendix


\section{Tropical \texorpdfstring{$y$-variables for $D_4$}{y-variables for D4}}\label{app.D4}

In this appendix let us explicitly compute the dilogarithm identity for $D_4$, 
as an illustration of the procedure to write down the quantum dilogarithm identity \eqref{eq.qdilog_tropical}.

Let us start with the quiver
\begin{align}
1,2,3\longleftarrow 4  \;,
\end{align}
which gives a skew-symmetric matrix 
\begin{align}
B=\begin{pmatrix}
0&0 & 0 & -1 \\
0&0 & 0 & -1 \\
0&0 & 0 & -1 \\
+1&+1 & +1 & 0 \\
\end{pmatrix} \;. 
\end{align}
Quantum $y$-variables $Y_1, Y_2, Y_3, Y_4$ then satisfy 
\begin{align}
Y_{a}Y_4 = q^2 Y_4 Y_{a} \;, \quad
Y_{a} Y_{b}=Y_{b} Y_{a} \;, \quad (a,b=1,2,3) \;.
\end{align}

Let us consider a mutation sequence $(1,2,3,4,1,3,2,4,1,2,3,4,1,2,3,4)$ of length $16$. This represents a closed loop in the exchange graph.
Indeed, 
the classical $y$-variables (with $q=1$) transform as
{\fontsize{5}{5}\selectfont
\begin{align}
\begin{array}{lll}
Y(0) = \left(Y_1, Y_2, Y_3, Y_4\right) \;,  
 \\
Y(1) = \left( \frac{1}{Y_ 1},Y_ 2,Y_ 3,\left(Y_ 1+1\right) Y_ 4 \right) \;,  
 \\
Y(2) = \left(  \frac{1}{Y_ 1},\frac{1}{Y_ 2},Y_ 3,\left(Y_ 1+1\right) \left(Y_ 2+1\right) Y_ 4 \right) \;,  
 \\
Y(3) =\left( \frac{1}{Y_ 1},\frac{1}{Y_ 2},\frac{1}{Y_ 3},\left(Y_ 1+1\right) \
\left(Y_ 2+1\right) \left(Y_ 3+1\right) Y_ 4 \right)\;,  
\\
Y(4) = \left(\frac{\left(Y_ 1+1\right) \left(Y_ 2+1\right) \left(Y_ 3+1\right) \ 
Y_ 4+1}{Y_ 1},\frac{\left(Y_ 1+1\right) \left(Y_ 2+1\right) \
\left(Y_ 3+1\right) Y_ 4+1}{Y_ 2},\frac{\left(Y_ 1+1\right) \
\left(Y_ 2+1\right) \left(Y_ 3+1\right) Y_ 4+1}{Y_3},\frac{1}{\left(Y_ 1+1\right) \left(Y_ 2+1\right) \left(Y_ 3+1\right) Y_ 4}\right) \;,  
\\
Y(5) = \left(\frac{Y_ 1}{\left(Y_ 1+1\right) \left(Y_ 2+1\right) \left(Y_ 3+1\right) Y_ 4+1},\frac{\left(Y_ 1+1\right) \left(Y_ 2+1\right) \
\left(Y_ 3+1\right) Y_ 4+1}{Y_ 2},\frac{\left(Y_ 1+1\right) \
\left(Y_ 2+1\right) \left(Y_ 3+1\right) Y_ 4+1}{Y_3},\frac{\left(Y_ 2+1\right) \left(Y_ 3+1\right) Y_ 4+1}{Y_ 1 \
\left(Y_ 2+1\right) \left(Y_ 3+1\right) Y_ 4}\right) \;,  
\\
Y(6) = \left(\frac{Y_ 1}{\left(Y_ 1+1\right) \left(Y_ 2+1\right) \left(Y_ 3+1\right) Y_ 4+1},\frac{Y_ 2}{\left(Y_ 1+1\right) \left(Y_ 2+1\right) \
\left(Y_ 3+1\right) Y_ 4+1},\frac{\left(Y_ 1+1\right) \left(Y_ 2+1\right) \left(Y_ 3+1\right) Y_ 4+1}{Y_ 3},\frac{\left(\left(Y_ 1+1\right) \left(Y_ 3+1\right) Y_ 4+1\right) \left(\left(Y_ 2+1\right) \
\left(Y_ 3+1\right) Y_ 4+1\right)}{Y_ 1 Y_ 2 \left(Y_ 3+1\right) \
Y_ 4}\right) \;,  
\\
Y(7) = \left(\frac{Y_ 1}{\left(Y_ 1+1\right) \left(Y_ 2+1\right) \left(Y_ 3+1\right) Y_ 4+1},\frac{Y_ 2}{\left(Y_ 1+1\right) \left(Y_ 2+1\right) 
\left(Y_ 3+1\right) Y_ 4+1},\frac{Y_ 3}{\left(Y_ 1+1\right) \
\left(Y_ 2+1\right) \left(Y_ 3+1\right) Y_4+1},\frac{\left(\left(Y_ 1+1\right) \left(Y_ 2+1\right) Y_ 4+1\right) \left(\left(Y_ 1+1\right) \left(Y_ 3+1\right) Y_ 4+1\right) \
\left(\left(Y_ 2+1\right) \left(Y_ 3+1\right) Y_ 4+1\right)}{Y_1
Y_ 2 Y_ 3 Y_ 4}\right)\;,  
\\
Y(8) =\left(\frac{\left(Y_ 1+1\right) \left(Y_ 2+1\right) \left(Y_ 3+1\right) \
Y_ 4^2+\left(Y_ 1+Y_ 2+Y_ 3+2\right) Y_ 4+1}{Y_ 2 Y_ 3 Y_4},\frac{\left(Y_ 1+1\right) \left(Y_ 2+1\right) \left(Y_ 3+1\right) Y_ 4^2+\left(Y_ 1+Y_ 2+Y_ 3+2\right) Y_ 4+1}{Y_ 1 Y_ 3 Y_4},\frac{\left(Y_ 1+1\right) \left(Y_ 2+1\right) \left(Y_ 3+1\right) Y_ 4^2+\left(Y_ 1+Y_ 2+Y_ 3+2\right) Y_ 4+1}{Y_ 1 Y_ 2 Y_4},\frac{Y_ 1 Y_ 2 Y_ 3 Y_ 4}{\left(\left(Y_ 1+1\right) \left(Y_ 2+1\
\right) Y_ 4+1\right) \left(\left(Y_ 1+1\right) \left(Y_ 3+1\right) Y_ 4+1\right) \left(\left(Y_ 2+1\right) \left(Y_ 3+1\right) \
Y_ 4+1\right)}\right) \;,  
 \\
Y(9) =\left(\frac{Y_ 2 Y_ 3 Y_ 4}{\left(Y_ 1+1\right) \left(Y_ 2+1\right) \
\left(Y_ 3+1\right) Y_ 4^2+\left(Y_ 1+Y_ 2+Y_ 3+2\right) Y_4+1},\frac{\left(Y_ 1+1\right) \left(Y_ 2+1\right) \left(Y_ 3+1\right) Y_ 4^2+\left(Y_ 1+Y_ 2+Y_ 3+2\right) Y_ 4+1}{Y_ 1 Y_ 3 Y_4},\frac{\left(Y_ 1+1\right) \left(Y_ 2+1\right) \left(Y_ 3+1\right) Y_ 4^2+\left(Y_ 1+Y_ 2+Y_ 3+2\right) Y_ 4+1}{Y_ 1 Y_ 2 Y_4},\frac{Y_ 1 \left(\left(Y_ 1+1\right) Y_ 4+1\right)}{\left(\left(Y_ 1+1\right) \left(Y_ 2+1\right) Y_ 4+1\right) 
\left(\left(Y_ 1+1\right) \left(Y_ 3+1\right) Y_ 4+1\right)}\right) \;,  
\\
Y(10) =\left(\frac{Y_ 2 Y_ 3 Y_ 4}{\left(Y_ 1+1\right) \left(Y_ 2+1\right) \
\left(Y_ 3+1\right) Y_ 4^2+\left(Y_ 1+Y_ 2+Y_ 3+2\right) Y_ 4+1},\frac{Y_ 1 Y_ 3 Y_ 4}{\left(Y_ 1+1\right) \left(Y_ 2+1\right) \
\left(Y_ 3+1\right) Y_ 4^2+\left(Y_ 1+Y_ 2+Y_ 3+2\right) Y_4+1},\frac{\left(Y_ 1+1\right) \left(Y_ 2+1\right) \left(Y_ 3+1\right) Y_ 4^2+\left(Y_ 1+Y_ 2+Y_ 3+2\right) Y_ 4+1}{Y_ 1 Y_ 2 Y_4},\frac{\left(\left(Y_ 1+1\right) Y_ 4+1\right) \left(\left(Y_ 2+1\
\right) Y_ 4+1\right)}{Y_ 3 Y_ 4 \left(\left(Y_ 1+1\right) 
\left(Y_ 2+1\right) Y_ 4+1\right)}\right) \;,  
\\
Y(11) =\left(
\frac{Y_ 2 Y_ 3 Y_ 4}{\left(Y_ 1+1\right) \left(Y_ 2+1\right) 
\left(Y_ 3+1\right) Y_ 4^2+\left(Y_ 1+Y_ 2+Y_ 3+2\right) Y_4+1},
\frac{Y_ 1 Y_ 3 Y_ 4}{\left(Y_ 1+1\right) \left(Y_ 2+1\right) 
\left(Y_ 3+1\right) Y_ 4^2+\left(Y_ 1+Y_ 2+Y_ 3+2\right) Y_4+1},
\frac{Y_ 1 Y_ 2 Y_ 4}{\left(Y_ 1+1\right) \left(Y_ 2+1\right) \
\left(Y_ 3+1\right) Y_ 4^2+\left(Y_ 1+Y_ 2+Y_ 3+2\right) Y_4+1},
\frac{\left(\left(Y_ 1+1\right) Y_ 4+1\right) \left(\left(Y_2+1\right) Y_ 4+1\right) \left(\left(Y_ 3+1\right) Y_ 4+1\right)}{Y_ 1 Y_ 2 Y_ 3 Y_ 4^2}
\right)  \;,  
\\
Y(12) =\left(\frac{Y_ 4+1}{Y_ 1 Y_ 4},\frac{Y_ 4+1}{Y_ 2 Y_ 4},\frac{Y_ 4+1}{Y_ 3 
Y_ 4},\frac{Y_ 1 Y_ 2 Y_ 3 Y_ 4^2}{\left(\left(Y_ 1+1\right) Y_ 4+1\right) \left(\left(Y_ 2+1\right) Y_ 4+1\right) \left(\left(Y_ 3+1\right) Y_ 4+1\right)}\right)  \;,  
\\
Y(13) =\left(\frac{Y_ 1 Y_ 4}{Y_ 4+1},\frac{Y_ 4+1}{Y_ 2 Y_ 4},\frac{Y_ 4+1}{Y_ 3 
Y_ 4},\frac{Y_ 2 Y_ 3 Y_ 4}{\left(\left(Y_ 2+1\right) Y_ 4+1\right) \
\left(\left(Y_ 3+1\right) Y_ 4+1\right)}\right) \;,  
\\
Y(14) =\left(\frac{Y_ 1 Y_ 4}{Y_ 4+1},\frac{Y_ 2 Y_ 4}{Y_ 4+1},\frac{Y_ 4+1}{Y_ 3 
Y_ 4},\frac{Y_ 3}{\left(Y_ 3+1\right) Y_ 4+1}\right)  \;,  
\\
Y(15) =\left(\frac{Y_ 1 Y_ 4}{Y_ 4+1},\frac{Y_ 2 Y_ 4}{Y_ 4+1},\frac{Y_ 3 Y_ 4}{Y_4+1},\frac{1}{Y_ 4}\right) \;,  %
\\
Y(16) =\left(Y_ 1,Y_ 2,Y_ 3,Y_4\right)  \;.
\\
\end{array}
\end{align}
}
The $c$-vectors $c_t:=c(y_{k_t}(t))$ and tropical signs $\epsilon_t$ are then computed to be
{\small
\begin{align}
\begin{array}{lll}
c_{1}=(1,0,0,0),&\epsilon_{1}=+\;, \\
c_{2}=(0,1,0,0),&\epsilon_{2}=+\;, \\
c_{3}=(0,0,1,0),&\epsilon_{3}=+\;, \\
c_{4}=(0,0,0,1),&\epsilon_{4}=+\;, \\
c_{5}=(1,0,0,0),&\epsilon_{5}=-\;, \\
c_{6}=(0,1,0,0),&\epsilon_{6}=-\;, \\
c_{7}=(0,0,1,0),&\epsilon_{7}=-\;, \\
c_{8}=(1,1,1,1),&\epsilon_{8}=-\;, \\
c_{9}=(0,1,1,1),&\epsilon_{9}=-\;, \\
c_{10}=(1,0,1,1),&\epsilon_{10}=-\;, \\
c_{11}=(1,1,0,1),&\epsilon_{11}=-\;, \\
c_{12}=(1,1,1,2),&\epsilon_{12}=-\;, \\
c_{13}=(1,0,0,1),&\epsilon_{13}=-\;, \\
c_{14}=(0,1,0,1),&\epsilon_{14}=-\;, \\
c_{15}=(0,0,1,1),&\epsilon_{15}=- \;,\\
c_{16}=(0,0,0,1),&\epsilon_{16}=- \;. \\
\end{array}
\end{align}
}
By substituting these data into the \eqref{eq.qdilog_tropical}, we obtain the $D_4$ quantum dilogarithm identity \eqref{D4_qdilog}.


\section{Explicit Proof of Conjecture for \texorpdfstring{$(A_3, A_3)$}{(A3,A3)} Quiver}\label{app.A3A3}

In this appendix we provide details of the computer-aided verification of the Conjecture \ref{conj1}
for the length $36$ loop of the exchange diagram of the $(A_3, A_3)$ quiver. 
The quiver is shown in Figure \ref{fig.A3A3quiver}
\begin{figure}[htbp]
\includegraphics[scale=0.4]{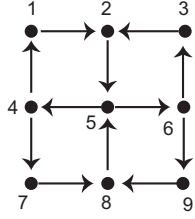}
\caption{Quiver associated with a pair of Lie algebra $(A_3, A_3)$.}
\label{fig.A3A3quiver}
\end{figure}
The length-$36$ loop of the exchange graph is given by the mutation sequence
\begin{align}
1, 3, 5, 7, 9, 2, 4, 6, 8, 1, 3, 5, 7, 9, 2, 4, 6, 8, 1, 3, 5, 7, 9, \
2, 4, 6, 8, 1, 3, 5, 7, 9, 2, 4, 6, 8  \;.
\end{align}
The associated quantum dilogarithm identity can be written as 
\begin{align}
\begin{split}
&
(1)(3)(5)(7)(123)(9)(45)(56)(789)(23)(12)(456)(89)(78)(2)(4)(6)(8)
 \\   
 & 
 = 
(2)(4)(6)(8)(14)(36)(258)(47)(69)(5
8)(147)(369)(25)(1)(3)(5)(7)(9)
\end{split}  
\label{long_id}
\end{align}
Here to save space we have dropped $\Psi_q$ from the notation:
for example $(789)$ represents $\Psi_q(789)$.

This relation can be shown by repeated use of $A_2$ and $A_1\times A_1$ identities.
It turns out that we need to use such identities approximately $280$ times, 
as shown below:

\input{longeq.tex}

\bibliographystyle{amsalpha}
\bibliography{pentagon}

\end{document}

%% file: longeq.tex
{\fontsize{5}{6}\selectfont 
\begin{align*}
&(1)(3)(5)(7)(123)(9)(45)(56)(789)(23)(12)(456)(89)(78)(2)(4)(6)(8)
 \\   
 & = 
(1)(3)(5)(7)(123)(45)(9)(56)(789)(23)(12)(456)(89)(78)(2)(4)(6)(8)
 \\   
 & = 
(1)(3)(5)(7)(123)(45)(56)(569)(9)(789)(23)(12)(456)(89)(78)(2)(4)(6)(8)
 \\   
 & = 
(1)(3)(5)(123)(7)(45)(56)(569)(9)(789)(23)(12)(456)(89)(78)(2)(4)(6)(8)
 \\   
 & = 
(1)(3)(5)(123)(45)(457)(7)(56)(569)(9)(789)(23)(12)(456)(89)(78)(2)(4)(6)(8)
 \\   
 & = 
(1)(3)(5)(123)(45)(457)(56)(7)(569)(9)(789)(23)(12)(456)(89)(78)(2)(4)(6)(8)
 \\   
 & = 
(1)(3)(5)(123)(45)(457)(56)(569)(7)(9)(789)(23)(12)(456)(89)(78)(2)(4)(6)(8)
 \\   
 & = 
(1)(3)(123)(1235)(5)(45)(457)(56)(569)(7)(9)(789)(23)(12)(456)(89)(78)(2)(4)(6)(8)
 \\   
 & = 
(1)(3)(123)(1235)(5)(45)(457)(56)(569)(7)(9)(789)(23)(12)(456)(89)(2)(78)(4)(6)(8)
 \\   
 & = 
(1)(3)(123)(1235)(5)(45)(457)(56)(569)(7)(9)(789)(23)(12)(456)(2)(89)(78)(4)(6)(8)
 \\   
 & = 
(1)(3)(123)(1235)(5)(45)(457)(56)(569)(7)(9)(789)(23)(12)(2)(2456)(456)(89)(78)(4)(6)(8)
 \\   
 & = 
(1)(3)(123)(1235)(5)(45)(457)(56)(569)(7)(9)(789)(23)(12)(2)(2456)(456)(78)(89)(4)(6)(8)
 \\   
 & = 
(1)(3)(123)(1235)(5)(45)(457)(56)(569)(7)(9)(789)(23)(12)(2)(2456)(78)(456)(89)(4)(6)(8)
 \\   
 & = 
(1)(3)(123)(1235)(5)(45)(457)(56)(569)(7)(9)(789)(23)(12)(2)(78)(2456)(456)(89)(4)(6)(8)
 \\   
 & = 
(1)(3)(123)(1235)(5)(45)(457)(56)(569)(7)(9)(789)(23)(12)(78)(2)(2456)(456)(89)(4)(6)(8)
 \\   
 & = 
(1)(3)(123)(1235)(5)(45)(457)(56)(569)(7)(9)(789)(23)(78)(12)(2)(2456)(456)(89)(4)(6)(8)
 \\   
 & = 
(1)(3)(123)(1235)(5)(45)(457)(56)(569)(7)(9)(789)(78)(23)(12)(2)(2456)(456)(89)(4)(6)(8)
 \\   
 & = 
(1)(3)(123)(1235)(5)(45)(457)(56)(569)(7)(78)(9)(23)(12)(2)(2456)(456)(89)(4)(6)(8)
 \\   
 & = 
(1)(3)(123)(1235)(5)(45)(457)(56)(569)(7)(78)(23)(9)(12)(2)(2456)(456)(89)(4)(6)(8)
 \\   
 & = 
(1)(3)(123)(1235)(5)(45)(457)(56)(569)(7)(78)(23)(12)(9)(2)(2456)(456)(89)(4)(6)(8)
 \\   
 & = 
(1)(3)(123)(1235)(5)(45)(457)(56)(569)(7)(78)(23)(12)(2)(9)(2456)(456)(89)(4)(6)(8)
 \\   
 & = 
(1)(3)(123)(1235)(5)(45)(457)(56)(569)(7)(78)(23)(12)(2)(2456)(24569)(9)(456)(89)(4)(6)(8)
 \\   
 & = 
(1)(3)(123)(1235)(5)(45)(457)(5
6)(569)(7)(78)(23)(12)(2)(2456)(2
4569)(456)(4569)(9)(89)(4)(6)(8)
 \\   
 & = 
(1)(3)(123)(1235)(5)(45)(457)(5
6)(569)(7)(78)(23)(12)(2)(2456)(2
4569)(456)(4569)(9)(89)(6)(4)(8)
 \\   
 & = 
(1)(3)(123)(1235)(5)(45)(457)(5
6)(569)(7)(78)(23)(12)(2)(2456)(2
4569)(456)(4569)(9)(89)(6)(8)(4)
 \\   
 & = 
(1)(3)(123)(1235)(5)(45)(457)(5
6)(569)(7)(78)(23)(12)(2)(2456)(2
4569)(456)(4569)(9)(89)(8)(6)(4)
 \\   
 & = 
(1)(3)(123)(1235)(5)(45)(457)(5
6)(569)(7)(78)(23)(12)(2)(2456)(2
4569)(456)(4569)(8)(9)(6)(4)
 \\   
 & = 
(1)(3)(123)(1235)(5)(45)(457)(5
6)(569)(7)(78)(23)(12)(2)(2456)(2
4569)(456)(4569)(8)(6)(69)(9)(4)
 \\   
 & = 
(1)(3)(123)(1235)(5)(45)(457)(5
6)(569)(7)(78)(23)(12)(2)(2456)(2
4569)(456)(4569)(8)(6)(69)(4)(9)
 \\   
 & = 
(1)(3)(123)(1235)(5)(45)(457)(5
6)(569)(7)(78)(23)(12)(2)(2456)(2
4569)(456)(8)(4569)(6)(69)(4)(9)
 \\   
 & = 
(1)(3)(123)(1235)(5)(45)(457)(5
6)(569)(7)(78)(23)(12)(2)(2456)(2
4569)(8)(4568)(456)(4569)(6)(6
9)(4)(9)
 \\   
 & = 
(1)(3)(123)(1235)(5)(45)(457)(5
6)(569)(7)(78)(23)(12)(2)(245
6)(8)(24569)(4568)(456)(456
9)(6)(69)(4)(9)
 \\   
 & = 
(1)(3)(123)(1235)(5)(45)(457)(5
6)(569)(7)(78)(23)(12)(2)(8)(2456
8)(2456)(24569)(4568)(456)(45
69)(6)(69)(4)(9)
 \\   
 & = 
(1)(3)(123)(1235)(5)(45)(457)(5
6)(569)(7)(78)(23)(12)(8)(2)(2456
8)(2456)(24569)(4568)(456)(45
69)(6)(69)(4)(9)
 \\   
 & = 
(1)(3)(123)(1235)(5)(45)(457)(5
6)(569)(7)(78)(23)(8)(12)(2)(2456
8)(2456)(24569)(4568)(456)(45
69)(6)(69)(4)(9)
 \\   
 & = 
(1)(3)(123)(1235)(5)(45)(457)(5
6)(569)(7)(78)(8)(23)(12)(2)(2456
8)(2456)(24569)(4568)(456)(45
69)(6)(69)(4)(9)
 \\   
 & = 
(1)(3)(123)(1235)(5)(45)(457)(5
6)(569)(8)(7)(23)(12)(2)(24568)(2
456)(24569)(4568)(456)(456
9)(6)(69)(4)(9)
 \\   
 & = 
(1)(3)(123)(1235)(5)(45)(457)(5
6)(569)(8)(23)(7)(12)(2)(24568)(2
456)(24569)(4568)(456)(456
9)(6)(69)(4)(9)
 \\   
 & = 
(1)(3)(123)(1235)(5)(45)(457)(5
6)(569)(8)(23)(12)(7)(2)(24568)(2
456)(24569)(4568)(456)(456
9)(6)(69)(4)(9)
 \\   
 & = 
(1)(3)(123)(1235)(5)(45)(457)(5
6)(569)(8)(23)(12)(2)(7)(24568)(2
456)(24569)(4568)(456)(456
9)(6)(69)(4)(9)
 \\   
 & = 
(1)(3)(123)(1235)(5)(45)(457)(5
6)(569)(8)(23)(12)(2)(24568)(7)(2
456)(24569)(4568)(456)(456
9)(6)(69)(4)(9)
 \\   
 & = 
(1)(3)(123)(1235)(5)(45)(457)(5
6)(569)(8)(23)(12)(2)(24568)(245
6)(24567)(7)(24569)(4568)(45
6)(4569)(6)(69)(4)(9)
 \\   
 & = 
(1)(3)(123)(1235)(5)(45)(457)(5
6)(569)(8)(23)(12)(2)(24568)(245
6)(24567)(24569)(245679)(7)(4
568)(456)(4569)(6)(69)(4)(9)
 \\   
 & = 
(1)(3)(123)(1235)(5)(45)(457)(5
6)(569)(8)(23)(12)(2)(24568)(245
6)(24567)(24569)(245679)(456
8)(7)(456)(4569)(6)(69)(4)(9)
 \\   
 & = 
(1)(3)(123)(1235)(5)(45)(457)(5
6)(569)(8)(23)(12)(2)(24568)(245
6)(24567)(24569)(245679)(456
8)(456)(4567)(7)(4569)(6)(6
9)(4)(9)
 \\   
 & = 
(1)(3)(123)(1235)(5)(45)(457)(5
6)(569)(8)(23)(12)(2)(24568)(245
6)(24567)(24569)(245679)(456
8)(456)(4567)(4569)(4567
9)(7)(6)(69)(4)(9)
 \\   
 & = 
(1)(3)(123)(1235)(5)(45)(457)(5
6)(569)(8)(23)(12)(2)(24568)(245
6)(24567)(24569)(245679)(456
8)(456)(4567)(4569)(4567
9)(6)(7)(69)(4)(9)
 \\   
 & = 
(1)(3)(123)(1235)(5)(45)(457)(5
6)(569)(8)(23)(12)(2)(24568)(245
6)(24567)(24569)(245679)(456
8)(456)(4567)(4569)(45679)(6)(6
9)(7)(4)(9)
 \\   
 & = 
(1)(3)(123)(1235)(5)(45)(457)(5
6)(569)(8)(23)(12)(2)(24568)(245
6)(24567)(24569)(245679)(456
8)(456)(4567)(4569)(45679)(6)(6
9)(4)(47)(7)(9)
 \\   
 & = 
(1)(3)(123)(1235)(5)(45)(457)(5
6)(569)(8)(23)(12)(2)(24568)(245
6)(24567)(24569)(245679)(456
8)(456)(4567)(4569)(4567
9)(6)(4)(69)(47)(7)(9)
 \\   
 & = 
(1)(3)(123)(1235)(5)(45)(457)(5
6)(569)(8)(23)(12)(2)(24568)(245
6)(24567)(24569)(245679)(456
8)(456)(4567)(4569)(4567
9)(4)(6)(69)(47)(7)(9)
 \\   
 & = 
(1)(3)(123)(1235)(5)(45)(457)(5
6)(569)(8)(23)(12)(2)(24568)(245
6)(24567)(24569)(245679)(456
8)(456)(4567)(4569)(4)(4567
9)(6)(69)(47)(7)(9)
 \\   
 & = 
(1)(3)(123)(1235)(5)(45)(457)(5
6)(8)(569)(23)(12)(2)(24568)(245
6)(24567)(24569)(245679)(456
8)(456)(4567)(4569)(4)(4567
9)(6)(69)(47)(7)(9)
 \\   
 & = 
(1)(3)(123)(1235)(5)(45)(457)(5
6)(8)(23)(569)(12)(2)(24568)(245
6)(24567)(24569)(245679)(456
8)(456)(4567)(4569)(4)(4567
9)(6)(69)(47)(7)(9)
 \\   
 & = 
(1)(3)(123)(1235)(5)(45)(457)(5
6)(8)(23)(12)(12569)(569)(2)(245
68)(2456)(24567)(24569)(2456
79)(4568)(456)(4567)(4569)(4)(4
5679)(6)(69)(47)(7)(9)
 \\   
 & = 
(1)(3)(123)(1235)(5)(45)(457)(5
6)(8)(23)(12)(12569)(2)(2569)(56
9)(24568)(2456)(24567)(2456
9)(245679)(4568)(456)(4567)(4
569)(4)(45679)(6)(69)(47)(7)(9)
 \\   
 & = 
(1)(3)(123)(1235)(5)(45)(457)(5
6)(8)(23)(12)(12569)(2)(2569)(24
568)(24556689)(569)(2456)(24
567)(24569)(245679)(4568)(45
6)(4567)(4569)(4)(45679)(6)(6
9)(47)(7)(9)
 \\   
 & = 
(1)(3)(123)(1235)(5)(45)(457)(5
6)(8)(23)(12)(12569)(2)(2569)(24
568)(24556689)(2456)(245566
9)(569)(24567)(24569)(24567
9)(4568)(456)(4567)(4569)(4)(45
679)(6)(69)(47)(7)(9)
 \\   
 & = 
(1)(3)(123)(1235)(5)(45)(457)(5
6)(8)(23)(12)(12569)(2)(2569)(24
568)(24556689)(2456)(245566
9)(24567)(24556679)(569)(245
69)(245679)(4568)(456)(456
7)(4569)(4)(45679)(6)(69)(4
7)(7)(9)
 \\   
 & = 
(1)(3)(123)(1235)(5)(45)(457)(5
6)(8)(23)(12)(12569)(2)(2569)(24
568)(24556689)(2456)(245566
9)(24567)(24556679)(24569)(5
69)(245679)(4568)(456)(456
7)(4569)(4)(45679)(6)(69)(4
7)(7)(9)
 \\   
 & = 
(1)(3)(123)(1235)(5)(45)(457)(5
6)(8)(23)(12)(12569)(2)(2569)(24
568)(24556689)(2456)(245566
9)(24567)(24556679)(24569)(2
45679)(569)(4568)(456)(456
7)(4569)(4)(45679)(6)(69)(4
7)(7)(9)
 \\   
 & = 
(1)(3)(123)(1235)(5)(45)(457)(5
6)(8)(23)(12)(12569)(2)(2569)(24
568)(24556689)(2456)(245566
9)(24567)(24556679)(24569)(2
45679)(4568)(569)(456)(456
7)(4569)(4)(45679)(6)(69)(4
7)(7)(9)
 \\   
 & = 
(1)(3)(123)(1235)(5)(45)(457)(5
6)(8)(23)(12)(12569)(2)(2569)(24
568)(24556689)(2456)(245566
9)(24567)(24556679)(24569)(2
45679)(4568)(456)(569)(456
7)(4569)(4)(45679)(6)(69)(4
7)(7)(9)
 \\   
 & = 
(1)(3)(123)(1235)(5)(45)(457)(5
6)(8)(23)(12)(12569)(2)(2569)(24
568)(24556689)(2456)(245566
9)(24567)(24556679)(24569)(2
45679)(4568)(456)(4567)(56
9)(4569)(4)(45679)(6)(69)(4
7)(7)(9)
 \\   
 & = 
(1)(3)(123)(1235)(5)(45)(457)(5
6)(8)(23)(12)(12569)(2)(2569)(24
568)(24556689)(2456)(245566
9)(24567)(24556679)(24569)(2
45679)(4568)(456)(4567)(4)(56
9)(45679)(6)(69)(47)(7)(9)
 \\   
 & = 
(1)(3)(123)(1235)(5)(45)(457)(5
6)(8)(23)(12)(12569)(2)(2569)(24
568)(24556689)(2456)(245566
9)(24567)(24556679)(24569)(2
45679)(4568)(456)(4)(4567)(56
9)(45679)(6)(69)(47)(7)(9)
 \\   
 & = 
(1)(3)(123)(1235)(5)(45)(457)(8)(5
68)(56)(23)(12)(12569)(2)(256
9)(24568)(24556689)(2456)(24
55669)(24567)(24556679)(245
69)(245679)(4568)(456)(4)(456
7)(569)(45679)(6)(69)(47)(7)(9)
 \\   
 & = 
(1)(3)(123)(1235)(5)(45)(457)(8)(5
68)(23)(56)(12)(12569)(2)(256
9)(24568)(24556689)(2456)(24
55669)(24567)(24556679)(245
69)(245679)(4568)(456)(4)(456
7)(569)(45679)(6)(69)(47)(7)(9)
 \\   
 & = 
(1)(3)(123)(1235)(5)(45)(457)(8)(5
68)(23)(12)(1256)(56)(1256
9)(2)(2569)(24568)(2455668
9)(2456)(2455669)(24567)(245
56679)(24569)(245679)(456
8)(456)(4)(4567)(569)(4567
9)(6)(69)(47)(7)(9)
 \\   
 & = 
(1)(3)(123)(1235)(5)(45)(457)(8)(5
68)(23)(12)(1256)(12569)(5
6)(2)(2569)(24568)(2455668
9)(2456)(2455669)(24567)(245
56679)(24569)(245679)(456
8)(456)(4)(4567)(569)(4567
9)(6)(69)(47)(7)(9)
 \\   
 & = 
(1)(3)(123)(1235)(5)(45)(457)(8)(5
68)(23)(12)(1256)(12569)(2)(25
6)(56)(2569)(24568)(2455668
9)(2456)(2455669)(24567)(245
56679)(24569)(245679)(456
8)(456)(4)(4567)(569)(4567
9)(6)(69)(47)(7)(9)
 \\   
 & = 
(1)(3)(123)(1235)(5)(45)(457)(8)(5
68)(23)(12)(1256)(12569)(2)(25
6)(2569)(56)(24568)(2455668
9)(2456)(2455669)(24567)(245
56679)(24569)(245679)(456
8)(456)(4)(4567)(569)(4567
9)(6)(69)(47)(7)(9)
 \\   
 & = 
(1)(3)(123)(1235)(5)(45)(457)(8)(5
68)(23)(12)(1256)(12569)(2)(25
6)(2569)(24568)(2455668)(5
6)(24556689)(2456)(245566
9)(24567)(24556679)(24569)(2
45679)(4568)(456)(4)(4567)(56
9)(45679)(6)(69)(47)(7)(9)
 \\   
 & = 
(1)(3)(123)(1235)(5)(45)(457)(8)(5
68)(23)(12)(1256)(12569)(2)(25
6)(2569)(24568)(2455668)(245
56689)(56)(2456)(2455669)(24
567)(24556679)(24569)(24567
9)(4568)(456)(4)(4567)(569)(45
679)(6)(69)(47)(7)(9)
 \\   
 & = 
(1)(3)(123)(1235)(5)(45)(457)(8)(5
68)(23)(12)(1256)(12569)(2)(25
6)(2569)(24568)(2455668)(245
56689)(2456)(56)(2455669)(24
567)(24556679)(24569)(24567
9)(4568)(456)(4)(4567)(569)(45
679)(6)(69)(47)(7)(9)
 \\   
 & = 
(1)(3)(123)(1235)(5)(45)(457)(8)(5
68)(23)(12)(1256)(12569)(2)(25
6)(2569)(24568)(2455668)(245
56689)(2456)(56)(2455669)(24
567)(24569)(24556679)(24567
9)(4568)(456)(4)(4567)(569)(45
679)(6)(69)(47)(7)(9)
 \\   
 & = 
(1)(3)(123)(1235)(5)(45)(457)(8)(5
68)(23)(12)(1256)(12569)(2)(25
6)(2569)(24568)(2455668)(245
56689)(2456)(56)(2455669)(24
569)(24567)(24556679)(24567
9)(4568)(456)(4)(4567)(569)(45
679)(6)(69)(47)(7)(9)
 \\   
 & = 
(1)(3)(123)(1235)(5)(45)(457)(8)(5
68)(23)(12)(1256)(12569)(2)(25
6)(2569)(24568)(2455668)(245
56689)(2456)(24569)(56)(2456
7)(24556679)(245679)(4568)(4
56)(4)(4567)(569)(45679)(6)(6
9)(47)(7)(9)
 \\   
 & = 
(1)(3)(123)(1235)(5)(45)(457)(8)(5
68)(23)(12)(1256)(12569)(2)(25
6)(2569)(24568)(2455668)(245
56689)(2456)(24569)(24567)(5
6)(24556679)(245679)(4568)(4
56)(4)(4567)(569)(45679)(6)(6
9)(47)(7)(9)
 \\   
 & = 
(1)(3)(123)(1235)(5)(45)(457)(8)(5
68)(23)(12)(1256)(12569)(2)(25
6)(2569)(24568)(2455668)(245
56689)(2456)(24569)(24567)(2
45679)(56)(4568)(456)(4)(456
7)(569)(45679)(6)(69)(47)(7)(9)
 \\   
 & = 
(1)(3)(123)(1235)(5)(45)(457)(8)(5
68)(23)(12)(1256)(12569)(2)(25
6)(2569)(24568)(2455668)(245
56689)(2456)(24569)(24567)(2
45679)(4568)(56)(456)(4)(456
7)(569)(45679)(6)(69)(47)(7)(9)
 \\   
 & = 
(1)(3)(123)(1235)(5)(45)(457)(8)(5
68)(23)(12)(1256)(12569)(2)(25
6)(2569)(24568)(2455668)(245
56689)(2456)(24569)(24567)(2
45679)(4568)(4)(56)(4567)(56
9)(45679)(6)(69)(47)(7)(9)
 \\   
 & = 
(1)(3)(123)(1235)(5)(45)(457)(8)(2
3)(568)(12)(1256)(12569)(2)(25
6)(2569)(24568)(2455668)(245
56689)(2456)(24569)(24567)(2
45679)(4568)(4)(56)(4567)(56
9)(45679)(6)(69)(47)(7)(9)
 \\   
 & = 
(1)(3)(123)(1235)(5)(45)(457)(8)(2
3)(12)(12568)(568)(1256)(1256
9)(2)(256)(2569)(24568)(24556
68)(24556689)(2456)(24569)(2
4567)(245679)(4568)(4)(56)(45
67)(569)(45679)(6)(69)(47)(7)(9)
 \\   
 & = 
(1)(3)(123)(1235)(5)(45)(457)(8)(2
3)(12)(12568)(1256)(568)(1256
9)(2)(256)(2569)(24568)(24556
68)(24556689)(2456)(24569)(2
4567)(245679)(4568)(4)(56)(45
67)(569)(45679)(6)(69)(47)(7)(9)
 \\   
 & = 
(1)(3)(123)(1235)(5)(45)(457)(8)(2
3)(12)(12568)(1256)(12569)(56
8)(2)(256)(2569)(24568)(24556
68)(24556689)(2456)(24569)(2
4567)(245679)(4568)(4)(56)(45
67)(569)(45679)(6)(69)(47)(7)(9)
 \\   
 & = 
(1)(3)(123)(1235)(5)(45)(457)(8)(2
3)(12)(12568)(1256)(12569)(2)(2
568)(568)(256)(2569)(24568)(2
455668)(24556689)(2456)(245
69)(24567)(245679)(4568)(4)(5
6)(4567)(569)(45679)(6)(69)(4
7)(7)(9)
 \\   
 & = 
(1)(3)(123)(1235)(5)(45)(457)(8)(2
3)(12)(12568)(1256)(12569)(2)(2
568)(256)(568)(2569)(24568)(2
455668)(24556689)(2456)(245
69)(24567)(245679)(4568)(4)(5
6)(4567)(569)(45679)(6)(69)(4
7)(7)(9)
 \\   
 & = 
(1)(3)(123)(1235)(5)(45)(457)(8)(2
3)(12)(12568)(1256)(12569)(2)(2
568)(256)(2569)(568)(24568)(2
455668)(24556689)(2456)(245
69)(24567)(245679)(4568)(4)(5
6)(4567)(569)(45679)(6)(69)(4
7)(7)(9)
 \\   
 & = 
(1)(3)(123)(1235)(5)(45)(457)(8)(2
3)(12)(12568)(1256)(12569)(2)(2
568)(256)(2569)(24568)(568)(2
455668)(24556689)(2456)(245
69)(24567)(245679)(4568)(4)(5
6)(4567)(569)(45679)(6)(69)(4
7)(7)(9)
 \\   
 & = 
(1)(3)(123)(1235)(5)(45)(457)(8)(2
3)(12)(12568)(1256)(12569)(2)(2
568)(256)(2569)(24568)(568)(2
455668)(2456)(24556689)(245
69)(24567)(245679)(4568)(4)(5
6)(4567)(569)(45679)(6)(69)(4
7)(7)(9)
 \\   
 & = 
(1)(3)(123)(1235)(5)(45)(457)(8)(2
3)(12)(12568)(1256)(12569)(2)(2
568)(256)(2569)(24568)(245
6)(568)(24556689)(24569)(245
67)(245679)(4568)(4)(56)(456
7)(569)(45679)(6)(69)(47)(7)(9)
 \\   
 & = 
(1)(3)(123)(1235)(5)(45)(457)(8)(2
3)(12)(12568)(1256)(12569)(2)(2
568)(256)(2569)(24568)(245
6)(24569)(568)(24567)(24567
9)(4568)(4)(56)(4567)(569)(4567
9)(6)(69)(47)(7)(9)
 \\   
 & = 
(1)(3)(123)(1235)(5)(45)(457)(8)(2
3)(12)(12568)(1256)(12569)(2)(2
568)(256)(2569)(24568)(245
6)(24569)(24567)(568)(24567
9)(4568)(4)(56)(4567)(569)(4567
9)(6)(69)(47)(7)(9)
 \\   
 & = 
(1)(3)(123)(1235)(5)(45)(457)(8)(2
3)(12)(12568)(1256)(12569)(2)(2
568)(256)(2569)(24568)(245
6)(24569)(24567)(245679)(56
8)(4568)(4)(56)(4567)(569)(4567
9)(6)(69)(47)(7)(9)
 \\   
 & = 
(1)(3)(123)(1235)(5)(45)(457)(8)(2
3)(12)(12568)(1256)(12569)(2)(2
568)(256)(2569)(24568)(245
6)(24569)(24567)(245679)(4)(5
68)(56)(4567)(569)(45679)(6)(6
9)(47)(7)(9)
 \\   
 & = 
(1)(3)(123)(1235)(5)(45)(457)(8)(2
3)(12)(12568)(1256)(12569)(2)(2
568)(256)(2569)(24568)(245
6)(24569)(24567)(4)(245679)(5
68)(56)(4567)(569)(45679)(6)(6
9)(47)(7)(9)
 \\   
 & = 
(1)(3)(123)(1235)(5)(45)(457)(8)(2
3)(12)(12568)(1256)(12569)(2)(2
568)(256)(2569)(24568)(245
6)(24569)(4)(24567)(245679)(5
68)(56)(4567)(569)(45679)(6)(6
9)(47)(7)(9)
 \\   
 & = 
(1)(3)(123)(1235)(5)(45)(457)(8)(2
3)(12)(12568)(1256)(12569)(2)(2
568)(256)(24568)(2569)(245
6)(24569)(4)(24567)(245679)(5
68)(56)(4567)(569)(45679)(6)(6
9)(47)(7)(9)
 \\   
 & = 
(1)(3)(123)(1235)(5)(45)(457)(8)(2
3)(12)(12568)(1256)(12569)(2)(2
568)(256)(24568)(2456)(256
9)(24569)(4)(24567)(245679)(5
68)(56)(4567)(569)(45679)(6)(6
9)(47)(7)(9)
 \\   
 & = 
(1)(3)(123)(1235)(5)(45)(457)(8)(2
3)(12)(12568)(1256)(12569)(2)(2
568)(256)(24568)(2456)(4)(256
9)(24567)(245679)(568)(56)(45
67)(569)(45679)(6)(69)(47)(7)(9)
 \\   
 & = 
(1)(3)(123)(1235)(5)(45)(457)(8)(2
3)(12)(12568)(1256)(12569)(2)(2
568)(24568)(256)(2456)(4)(256
9)(24567)(245679)(568)(56)(45
67)(569)(45679)(6)(69)(47)(7)(9)
 \\   
 & = 
(1)(3)(123)(1235)(5)(45)(457)(8)(2
3)(12)(12568)(1256)(12569)(2)(2
568)(24568)(4)(256)(2569)(245
67)(245679)(568)(56)(4567)(56
9)(45679)(6)(69)(47)(7)(9)
 \\   
 & = 
(1)(3)(123)(1235)(5)(45)(457)(8)(2
3)(12)(12568)(1256)(1256
9)(2)(4)(2568)(256)(2569)(2456
7)(245679)(568)(56)(4567)(56
9)(45679)(6)(69)(47)(7)(9)
 \\   
 & = 
(1)(3)(123)(1235)(5)(45)(457)(8)(2
3)(12)(12568)(1256)(2)(1256
9)(4)(2568)(256)(2569)(24567)(2
45679)(568)(56)(4567)(569)(45
679)(6)(69)(47)(7)(9)
 \\   
 & = 
(1)(3)(123)(1235)(5)(45)(457)(8)(2
3)(12)(12568)(2)(1256)(1256
9)(4)(2568)(256)(2569)(24567)(2
45679)(568)(56)(4567)(569)(45
679)(6)(69)(47)(7)(9)
 \\   
 & = 
(1)(3)(123)(1235)(5)(45)(457)(8)(2
3)(12)(2)(12568)(1256)(1256
9)(4)(2568)(256)(2569)(24567)(2
45679)(568)(56)(4567)(569)(45
679)(6)(69)(47)(7)(9)
 \\   
 & = 
(1)(3)(123)(1235)(5)(45)(457)(8)(2
3)(12)(2)(12568)(1256)(4)(1256
9)(2568)(256)(2569)(24567)(24
5679)(568)(56)(4567)(569)(456
79)(6)(69)(47)(7)(9)
 \\   
 & = 
(1)(3)(123)(1235)(5)(45)(457)(8)(2
3)(12)(2)(12568)(4)(1256)(1256
9)(2568)(256)(2569)(24567)(24
5679)(568)(56)(4567)(569)(456
79)(6)(69)(47)(7)(9)
 \\   
 & = 
(1)(3)(123)(1235)(5)(45)(457)(8)(2
3)(12)(2)(4)(12568)(1256)(1256
9)(2568)(256)(2569)(24567)(24
5679)(568)(56)(4567)(569)(456
79)(6)(69)(47)(7)(9)
 \\   
 & = 
(1)(3)(123)(1235)(5)(45)(457)(8)(2
3)(12)(4)(2)(12568)(1256)(1256
9)(2568)(256)(2569)(24567)(24
5679)(568)(56)(4567)(569)(456
79)(6)(69)(47)(7)(9)
 \\   
 & = 
(1)(3)(123)(1235)(5)(45)(457)(8)(2
3)(4)(124)(12)(2)(12568)(1256)(1
2569)(2568)(256)(2569)(2456
7)(245679)(568)(56)(4567)(56
9)(45679)(6)(69)(47)(7)(9)
 \\   
 & = 
(1)(3)(123)(1235)(5)(45)(45
7)(8)(4)(23)(124)(12)(2)(12568)(1
256)(12569)(2568)(256)(256
9)(24567)(245679)(568)(56)(45
67)(569)(45679)(6)(69)(47)(7)(9)
 \\   
 & = 
(1)(3)(123)(1235)(5)(45)(45
7)(4)(8)(23)(124)(12)(2)(12568)(1
256)(12569)(2568)(256)(256
9)(24567)(245679)(568)(56)(45
67)(569)(45679)(6)(69)(47)(7)(9)
 \\   
 & = 
(1)(3)(123)(1235)(5)(45)(4)(45
7)(8)(23)(124)(12)(2)(12568)(125
6)(12569)(2568)(256)(2569)(24
567)(245679)(568)(56)(4567)(5
69)(45679)(6)(69)(47)(7)(9)
 \\   
 & = 
(1)(3)(123)(1235)(4)(5)(457)(8)(2
3)(124)(12)(2)(12568)(1256)(125
69)(2568)(256)(2569)(24567)(2
45679)(568)(56)(4567)(569)(45
679)(6)(69)(47)(7)(9)
 \\   
 & = 
(1)(3)(123)(4)(1235)(5)(457)(8)(2
3)(124)(12)(2)(12568)(1256)(125
69)(2568)(256)(2569)(24567)(2
45679)(568)(56)(4567)(569)(45
679)(6)(69)(47)(7)(9)
 \\   
 & = 
(1)(3)(4)(1234)(123)(1235)(5)(45
7)(8)(23)(124)(12)(2)(12568)(125
6)(12569)(2568)(256)(2569)(24
567)(245679)(568)(56)(4567)(5
69)(45679)(6)(69)(47)(7)(9)
 \\   
 & = 
(1)(4)(3)(1234)(123)(1235)(5)(45
7)(8)(23)(124)(12)(2)(12568)(125
6)(12569)(2568)(256)(2569)(24
567)(245679)(568)(56)(4567)(5
69)(45679)(6)(69)(47)(7)(9)
 \\   
 & = 
(4)(14)(1)(3)(1234)(123)(123
5)(5)(457)(8)(23)(124)(12)(2)(125
68)(1256)(12569)(2568)(256)(2
569)(24567)(245679)(568)(5
6)(4567)(569)(45679)(6)(69)(4
7)(7)(9)
 \\   
 & = 
(4)(14)(1)(3)(1234)(123)(123
5)(5)(457)(8)(124)(23)(12)(2)(125
68)(1256)(12569)(2568)(256)(2
569)(24567)(245679)(568)(5
6)(4567)(569)(45679)(6)(69)(4
7)(7)(9)
 \\   
 & = 
(4)(14)(1)(3)(1234)(123)(123
5)(5)(457)(124)(8)(23)(12)(2)(125
68)(1256)(12569)(2568)(256)(2
569)(24567)(245679)(568)(5
6)(4567)(569)(45679)(6)(69)(4
7)(7)(9)
 \\   
 & = 
(4)(14)(1)(3)(1234)(123)(123
5)(5)(124)(457)(8)(23)(12)(2)(125
68)(1256)(12569)(2568)(256)(2
569)(24567)(245679)(568)(5
6)(4567)(569)(45679)(6)(69)(4
7)(7)(9)
 \\   
 & = 
(4)(14)(1)(3)(1234)(123)(1235)(12
4)(5)(457)(8)(23)(12)(2)(12568)(1
256)(12569)(2568)(256)(256
9)(24567)(245679)(568)(56)(45
67)(569)(45679)(6)(69)(47)(7)(9)
 \\   
 & = 
(4)(14)(1)(3)(1234)(123)(124)(123
5)(5)(457)(8)(23)(12)(2)(12568)(1
256)(12569)(2568)(256)(256
9)(24567)(245679)(568)(56)(45
67)(569)(45679)(6)(69)(47)(7)(9)
 \\   
 & = 
(4)(14)(1)(3)(1234)(124)(123)(123
5)(5)(457)(8)(23)(12)(2)(12568)(1
256)(12569)(2568)(256)(256
9)(24567)(245679)(568)(56)(45
67)(569)(45679)(6)(69)(47)(7)(9)
 \\   
 & = 
(4)(14)(1)(124)(3)(123)(1235)(5)(4
57)(8)(23)(12)(2)(12568)(1256)(1
2569)(2568)(256)(2569)(2456
7)(245679)(568)(56)(4567)(56
9)(45679)(6)(69)(47)(7)(9)
 \\   
 & = 
(4)(14)(1)(124)(3)(123)(1235)(5)(4
57)(8)(12)(23)(2)(12568)(1256)(1
2569)(2568)(256)(2569)(2456
7)(245679)(568)(56)(4567)(56
9)(45679)(6)(69)(47)(7)(9)
 \\   
 & = 
(4)(14)(1)(124)(3)(123)(1235)(5)(4
57)(12)(8)(23)(2)(12568)(1256)(1
2569)(2568)(256)(2569)(2456
7)(245679)(568)(56)(4567)(56
9)(45679)(6)(69)(47)(7)(9)
 \\   
 & = 
(4)(14)(1)(124)(3)(123)(1235)(5)(1
2)(457)(8)(23)(2)(12568)(1256)(1
2569)(2568)(256)(2569)(2456
7)(245679)(568)(56)(4567)(56
9)(45679)(6)(69)(47)(7)(9)
 \\   
 & = 
(4)(14)(1)(124)(3)(123)(1235)(1
2)(125)(5)(457)(8)(23)(2)(1256
8)(1256)(12569)(2568)(256)(25
69)(24567)(245679)(568)(56)(4
567)(569)(45679)(6)(69)(47)(7)(9)
 \\   
 & = 
(4)(14)(1)(124)(3)(123)(12)(123
5)(125)(5)(457)(8)(23)(2)(1256
8)(1256)(12569)(2568)(256)(25
69)(24567)(245679)(568)(56)(4
567)(569)(45679)(6)(69)(47)(7)(9)
 \\   
 & = 
(4)(14)(1)(124)(12)(3)(1235)(12
5)(5)(457)(8)(23)(2)(12568)(125
6)(12569)(2568)(256)(2569)(24
567)(245679)(568)(56)(4567)(5
69)(45679)(6)(69)(47)(7)(9)
 \\   
 & = 
(4)(14)(1)(124)(12)(125)(3)(5)(45
7)(8)(23)(2)(12568)(1256)(1256
9)(2568)(256)(2569)(24567)(24
5679)(568)(56)(4567)(569)(456
79)(6)(69)(47)(7)(9)
 \\   
 & = 
(4)(14)(1)(124)(12)(125)(5)(3)(45
7)(8)(23)(2)(12568)(1256)(1256
9)(2568)(256)(2569)(24567)(24
5679)(568)(56)(4567)(569)(456
79)(6)(69)(47)(7)(9)
 \\   
 & = 
(4)(14)(1)(124)(12)(125)(5)(45
7)(3)(8)(23)(2)(12568)(1256)(125
69)(2568)(256)(2569)(24567)(2
45679)(568)(56)(4567)(569)(45
679)(6)(69)(47)(7)(9)
 \\   
 & = 
(4)(14)(1)(124)(12)(125)(5)(45
7)(8)(3)(23)(2)(12568)(1256)(125
69)(2568)(256)(2569)(24567)(2
45679)(568)(56)(4567)(569)(45
679)(6)(69)(47)(7)(9)
 \\   
 & = 
(4)(14)(1)(124)(12)(125)(5)(45
7)(8)(2)(3)(12568)(1256)(1256
9)(2568)(256)(2569)(24567)(24
5679)(568)(56)(4567)(569)(456
79)(6)(69)(47)(7)(9)
 \\   
 & = 
(4)(14)(1)(124)(12)(125)(5)(45
7)(8)(2)(12568)(3)(1256)(1256
9)(2568)(256)(2569)(24567)(24
5679)(568)(56)(4567)(569)(456
79)(6)(69)(47)(7)(9)
 \\   
 & = 
(4)(14)(1)(124)(12)(125)(5)(45
7)(8)(2)(12568)(1256)(3)(1256
9)(2568)(256)(2569)(24567)(24
5679)(568)(56)(4567)(569)(456
79)(6)(69)(47)(7)(9)
 \\   
 & = 
(4)(14)(1)(124)(12)(125)(5)(45
7)(8)(2)(12568)(1256)(1256
9)(3)(2568)(256)(2569)(24567)(2
45679)(568)(56)(4567)(569)(45
679)(6)(69)(47)(7)(9)
 \\   
 & = 
(4)(14)(1)(124)(12)(125)(5)(45
7)(8)(2)(12568)(1256)(12569)(25
68)(3)(256)(2569)(24567)(2456
79)(568)(56)(4567)(569)(4567
9)(6)(69)(47)(7)(9)
 \\   
 & = 
(4)(14)(1)(124)(12)(125)(5)(45
7)(8)(2)(12568)(1256)(12569)(25
68)(256)(3)(2569)(24567)(2456
79)(568)(56)(4567)(569)(4567
9)(6)(69)(47)(7)(9)
 \\   
 & = 
(4)(14)(1)(124)(12)(125)(5)(45
7)(8)(2)(12568)(1256)(12569)(25
68)(256)(2569)(3)(24567)(2456
79)(568)(56)(4567)(569)(4567
9)(6)(69)(47)(7)(9)
 \\   
 & = 
(4)(14)(1)(124)(12)(125)(5)(45
7)(8)(2)(12568)(1256)(12569)(25
68)(256)(2569)(24567)(3)(2456
79)(568)(56)(4567)(569)(4567
9)(6)(69)(47)(7)(9)
 \\   
 & = 
(4)(14)(1)(124)(12)(125)(5)(45
7)(8)(2)(12568)(1256)(12569)(25
68)(256)(2569)(24567)(24567
9)(3)(568)(56)(4567)(569)(4567
9)(6)(69)(47)(7)(9)
 \\   
 & = 
(4)(14)(1)(124)(12)(125)(5)(45
7)(8)(2)(12568)(1256)(12569)(25
68)(256)(2569)(24567)(24567
9)(568)(3568)(3)(56)(4567)(56
9)(45679)(6)(69)(47)(7)(9)
 \\   
 & = 
(4)(14)(1)(124)(12)(125)(5)(45
7)(8)(2)(12568)(1256)(12569)(25
68)(256)(2569)(24567)(24567
9)(568)(3568)(56)(356)(3)(456
7)(569)(45679)(6)(69)(47)(7)(9)
 \\   
 & = 
(4)(14)(1)(124)(12)(125)(5)(45
7)(8)(2)(12568)(1256)(12569)(25
68)(256)(2569)(24567)(24567
9)(568)(3568)(56)(356)(4567)(34
567)(3)(569)(45679)(6)(69)(4
7)(7)(9)
 \\   
 & = 
(4)(14)(1)(124)(12)(125)(5)(45
7)(8)(2)(12568)(1256)(12569)(25
68)(256)(2569)(24567)(24567
9)(568)(3568)(56)(356)(4567)(34
567)(569)(3569)(3)(45679)(6)(6
9)(47)(7)(9)
 \\   
 & = 
(4)(14)(1)(124)(12)(125)(5)(45
7)(8)(2)(12568)(1256)(12569)(25
68)(256)(2569)(24567)(24567
9)(568)(3568)(56)(356)(4567)(34
567)(569)(3569)(45679)(34567
9)(3)(6)(69)(47)(7)(9)
 \\   
 & = 
(4)(14)(1)(124)(12)(125)(5)(45
7)(8)(2)(12568)(1256)(12569)(25
68)(256)(2569)(24567)(24567
9)(568)(3568)(56)(356)(4567)(34
567)(569)(3569)(45679)(34567
9)(6)(36)(3)(69)(47)(7)(9)
 \\   
 & = 
(4)(14)(1)(124)(12)(125)(5)(45
7)(8)(2)(12568)(1256)(12569)(25
68)(256)(2569)(24567)(24567
9)(568)(3568)(56)(356)(4567)(34
567)(569)(3569)(45679)(34567
9)(6)(36)(69)(369)(3)(47)(7)(9)
 \\   
 & = 
(4)(14)(1)(124)(12)(125)(5)(45
7)(8)(2)(12568)(1256)(12569)(25
68)(256)(2569)(24567)(24567
9)(568)(3568)(56)(356)(4567)(34
567)(569)(3569)(45679)(34567
9)(6)(36)(69)(369)(47)(3)(7)(9)
 \\   
 & = 
(4)(14)(1)(124)(12)(125)(5)(45
7)(8)(2)(12568)(1256)(12569)(25
68)(256)(2569)(24567)(24567
9)(568)(3568)(56)(356)(4567)(34
567)(569)(3569)(45679)(34567
9)(6)(36)(69)(47)(369)(3)(7)(9)
 \\   
 & = 
(4)(14)(1)(124)(12)(125)(5)(45
7)(8)(2)(12568)(1256)(12569)(25
68)(256)(2569)(24567)(24567
9)(568)(3568)(56)(356)(4567)(34
567)(569)(3569)(45679)(34567
9)(6)(36)(47)(69)(369)(3)(7)(9)
 \\   
 & = 
(4)(14)(1)(124)(12)(125)(5)(45
7)(8)(2)(12568)(1256)(12569)(25
68)(256)(2569)(24567)(24567
9)(568)(3568)(56)(356)(4567)(34
567)(569)(3569)(45679)(34567
9)(6)(47)(36)(69)(369)(3)(7)(9)
 \\   
 & = 
(4)(14)(1)(124)(12)(125)(5)(45
7)(8)(2)(12568)(1256)(12569)(25
68)(256)(2569)(24567)(24567
9)(568)(3568)(56)(356)(4567)(34
567)(569)(3569)(45679)(34567
9)(47)(6)(36)(69)(369)(3)(7)(9)
 \\   
 & = 
(4)(14)(1)(124)(12)(125)(5)(45
7)(8)(2)(12568)(1256)(12569)(25
68)(256)(2569)(24567)(24567
9)(568)(3568)(56)(356)(4567)(34
567)(569)(45679)(3569)(34567
9)(47)(6)(36)(69)(369)(3)(7)(9)
 \\   
 & = 
(4)(14)(1)(124)(12)(125)(5)(45
7)(8)(2)(12568)(1256)(12569)(25
68)(256)(2569)(24567)(24567
9)(568)(3568)(56)(356)(4567)(34
567)(569)(45679)(47)(356
9)(6)(36)(69)(369)(3)(7)(9)
 \\   
 & = 
(4)(14)(1)(124)(12)(125)(5)(45
7)(8)(2)(12568)(1256)(12569)(25
68)(256)(2569)(24567)(24567
9)(568)(3568)(56)(356)(4567)(34
567)(47)(569)(3569)(6)(36)(69)(3
69)(3)(7)(9)
 \\   
 & = 
(4)(14)(1)(124)(12)(125)(5)(45
7)(8)(2)(12568)(1256)(12569)(25
68)(256)(2569)(24567)(24567
9)(568)(3568)(56)(4567)(356)(34
567)(47)(569)(3569)(6)(36)(69)(3
69)(3)(7)(9)
 \\   
 & = 
(4)(14)(1)(124)(12)(125)(5)(45
7)(8)(2)(12568)(1256)(12569)(25
68)(256)(2569)(24567)(24567
9)(568)(3568)(56)(4567)(47)(35
6)(569)(3569)(6)(36)(69)(36
9)(3)(7)(9)
 \\   
 & = 
(4)(14)(1)(124)(12)(125)(5)(45
7)(8)(2)(12568)(1256)(12569)(25
68)(256)(2569)(24567)(24567
9)(568)(3568)(47)(56)(356)(56
9)(3569)(6)(36)(69)(369)(3)(7)(9)
 \\   
 & = 
(4)(14)(1)(124)(12)(125)(5)(45
7)(8)(2)(12568)(1256)(12569)(25
68)(256)(2569)(24567)(24567
9)(568)(47)(3568)(56)(356)(56
9)(3569)(6)(36)(69)(369)(3)(7)(9)
 \\   
 & = 
(4)(14)(1)(124)(12)(125)(5)(45
7)(8)(2)(12568)(1256)(12569)(25
68)(256)(2569)(24567)(24567
9)(47)(568)(3568)(56)(356)(56
9)(3569)(6)(36)(69)(369)(3)(7)(9)
 \\   
 & = 
(4)(14)(1)(124)(12)(125)(5)(45
7)(8)(2)(12568)(1256)(12569)(25
68)(256)(24567)(2569)(24567
9)(47)(568)(3568)(56)(356)(56
9)(3569)(6)(36)(69)(369)(3)(7)(9)
 \\   
 & = 
(4)(14)(1)(124)(12)(125)(5)(45
7)(8)(2)(12568)(1256)(12569)(25
68)(256)(24567)(47)(2569)(56
8)(3568)(56)(356)(569)(356
9)(6)(36)(69)(369)(3)(7)(9)
 \\   
 & = 
(4)(14)(1)(124)(12)(125)(5)(45
7)(8)(2)(12568)(1256)(12569)(25
68)(47)(256)(2569)(568)(356
8)(56)(356)(569)(3569)(6)(36)(6
9)(369)(3)(7)(9)
 \\   
 & = 
(4)(14)(1)(124)(12)(125)(5)(45
7)(8)(2)(12568)(1256)(12569)(4
7)(2568)(256)(2569)(568)(356
8)(56)(356)(569)(3569)(6)(36)(6
9)(369)(3)(7)(9)
 \\   
 & = 
(4)(14)(1)(124)(12)(125)(5)(45
7)(8)(2)(12568)(1256)(47)(1256
9)(2568)(256)(2569)(568)(356
8)(56)(356)(569)(3569)(6)(36)(6
9)(369)(3)(7)(9)
 \\   
 & = 
(4)(14)(1)(124)(12)(125)(5)(45
7)(8)(2)(12568)(47)(1256)(1256
9)(2568)(256)(2569)(568)(356
8)(56)(356)(569)(3569)(6)(36)(6
9)(369)(3)(7)(9)
 \\   
 & = 
(4)(14)(1)(124)(12)(125)(5)(45
7)(8)(2)(47)(1245678)(12568)(12
56)(12569)(2568)(256)(2569)(5
68)(3568)(56)(356)(569)(356
9)(6)(36)(69)(369)(3)(7)(9)
 \\   
 & = 
(4)(14)(1)(124)(12)(125)(5)(45
7)(8)(47)(2)(1245678)(12568)(12
56)(12569)(2568)(256)(2569)(5
68)(3568)(56)(356)(569)(356
9)(6)(36)(69)(369)(3)(7)(9)
 \\   
 & = 
(4)(14)(1)(124)(12)(125)(5)(457)(4
7)(478)(8)(2)(1245678)(12568)(1
256)(12569)(2568)(256)(256
9)(568)(3568)(56)(356)(569)(356
9)(6)(36)(69)(369)(3)(7)(9)
 \\   
 & = 
(4)(14)(1)(124)(12)(125)(47)(5)(47
8)(8)(2)(1245678)(12568)(125
6)(12569)(2568)(256)(2569)(56
8)(3568)(56)(356)(569)(356
9)(6)(36)(69)(369)(3)(7)(9)
 \\   
 & = 
(4)(14)(1)(124)(12)(125)(47)(47
8)(5)(8)(2)(1245678)(12568)(125
6)(12569)(2568)(256)(2569)(56
8)(3568)(56)(356)(569)(356
9)(6)(36)(69)(369)(3)(7)(9)
 \\   
 & = 
(4)(14)(1)(124)(12)(125)(47)(47
8)(8)(58)(5)(2)(1245678)(1256
8)(1256)(12569)(2568)(256)(25
69)(568)(3568)(56)(356)(569)(35
69)(6)(36)(69)(369)(3)(7)(9)
 \\   
 & = 
(4)(14)(1)(124)(12)(125)(47)(47
8)(8)(58)(2)(25)(5)(1245678)(125
68)(1256)(12569)(2568)(256)(2
569)(568)(3568)(56)(356)(569)(3
569)(6)(36)(69)(369)(3)(7)(9)
 \\   
 & = 
(4)(14)(1)(124)(12)(125)(47)(47
8)(8)(58)(2)(25)(1245678)(5)(125
68)(1256)(12569)(2568)(256)(2
569)(568)(3568)(56)(356)(569)(3
569)(6)(36)(69)(369)(3)(7)(9)
 \\   
 & = 
(4)(14)(1)(124)(12)(125)(47)(47
8)(8)(58)(2)(25)(1245678)(1256
8)(125568)(5)(1256)(12569)(25
68)(256)(2569)(568)(3568)(56)(3
56)(569)(3569)(6)(36)(69)(36
9)(3)(7)(9)
 \\   
 & = 
(4)(14)(1)(124)(12)(125)(47)(47
8)(8)(58)(2)(25)(1245678)(1256
8)(125568)(1256)(5)(12569)(25
68)(256)(2569)(568)(3568)(56)(3
56)(569)(3569)(6)(36)(69)(36
9)(3)(7)(9)
 \\   
 & = 
(4)(14)(1)(124)(12)(125)(47)(47
8)(8)(58)(2)(25)(1245678)(1256
8)(125568)(1256)(12569)(5)(25
68)(256)(2569)(568)(3568)(56)(3
56)(569)(3569)(6)(36)(69)(36
9)(3)(7)(9)
 \\   
 & = 
(4)(14)(1)(124)(12)(125)(47)(47
8)(8)(58)(2)(25)(1245678)(1256
8)(125568)(1256)(12569)(256
8)(25568)(5)(256)(2569)(568)(35
68)(56)(356)(569)(3569)(6)(36)(6
9)(369)(3)(7)(9)
 \\   
 & = 
(4)(14)(1)(124)(12)(125)(47)(47
8)(8)(58)(2)(25)(1245678)(1256
8)(125568)(1256)(12569)(256
8)(25568)(256)(5)(2569)(568)(35
68)(56)(356)(569)(3569)(6)(36)(6
9)(369)(3)(7)(9)
 \\   
 & = 
(4)(14)(1)(124)(12)(125)(47)(47
8)(8)(58)(2)(25)(1245678)(1256
8)(125568)(1256)(12569)(256
8)(25568)(256)(2569)(5)(568)(35
68)(56)(356)(569)(3569)(6)(36)(6
9)(369)(3)(7)(9)
 \\   
 & = 
(4)(14)(1)(124)(12)(125)(47)(47
8)(8)(58)(2)(25)(1245678)(1256
8)(125568)(1256)(12569)(256
8)(25568)(256)(2569)(568)(5)(35
68)(56)(356)(569)(3569)(6)(36)(6
9)(369)(3)(7)(9)
 \\   
 & = 
(4)(14)(1)(124)(12)(125)(47)(47
8)(8)(58)(2)(25)(1245678)(1256
8)(125568)(1256)(12569)(256
8)(25568)(256)(2569)(568)(356
8)(5)(56)(356)(569)(3569)(6)(3
6)(69)(369)(3)(7)(9)
 \\   
 & = 
(4)(14)(1)(124)(12)(125)(47)(47
8)(8)(58)(2)(25)(1245678)(1256
8)(125568)(1256)(12569)(256
8)(25568)(256)(2569)(568)(356
8)(5)(56)(356)(569)(6)(35669)(35
69)(36)(69)(369)(3)(7)(9)
 \\   
 & = 
(4)(14)(1)(124)(12)(125)(47)(47
8)(8)(58)(2)(25)(1245678)(1256
8)(125568)(1256)(12569)(256
8)(25568)(256)(2569)(568)(356
8)(5)(56)(356)(6)(569)(35669)(35
69)(36)(69)(369)(3)(7)(9)
 \\   
 & = 
(4)(14)(1)(124)(12)(125)(47)(47
8)(8)(58)(2)(25)(1245678)(1256
8)(125568)(1256)(12569)(256
8)(25568)(256)(2569)(568)(356
8)(5)(56)(6)(356)(569)(35669)(35
69)(36)(69)(369)(3)(7)(9)
 \\   
 & = 
(4)(14)(1)(124)(12)(125)(47)(47
8)(8)(58)(2)(25)(1245678)(1256
8)(125568)(1256)(12569)(256
8)(25568)(256)(2569)(568)(356
8)(5)(56)(6)(356)(569)(35669)(3
6)(3569)(69)(369)(3)(7)(9)
 \\   
 & = 
(4)(14)(124)(1)(12)(125)(47)(47
8)(8)(58)(2)(25)(1245678)(1256
8)(125568)(1256)(12569)(256
8)(25568)(256)(2569)(568)(356
8)(5)(56)(6)(356)(569)(35669)(3
6)(3569)(69)(369)(3)(7)(9)
 \\   
 & = 
(4)(14)(124)(1)(12)(125)(47)(478)(8)(2)(258)(58)(25)(1245678)(1
2568)(125568)(1256)(12569)(2568)(25568)(256)(2569)(568)(3
568)(5)(56)(6)(356)(569)(35669)(36)(3569)(69)(369)(3)(7)(9)
 \\   
 & = 
(4)(14)(124)(1)(12)(125)(47)(478)(2)(8)(258)(58)(25)(1245678)(1
2568)(125568)(1256)(12569)(2568)(25568)(256)(2569)(568)(3
568)(5)(56)(6)(356)(569)(35669)(36)(3569)(69)(369)(3)(7)(9)
 \\   
 & = 
(4)(14)(124)(1)(12)(125)(47)(2)(478)(8)(258)(58)(25)(1245678)(125
68)(125568)(1256)(12569)(2568)(25568)(256)(2569)(568)(356
8)(5)(56)(6)(356)(569)(35669)(36)(3569)(69)(369)(3)(7)(9)
 \\   
 & = 
(4)(14)(124)(1)(12)(125)(2)(47)(478)(8)(258)(58)(25)(1245678)(125
68)(125568)(1256)(12569)(2568)(25568)(256)(2569)(568)(356
8)(5)(56)(6)(356)(569)(35669)(36)(3569)(69)(369)(3)(7)(9)
 \\   
 & = 
(4)(14)(124)(1)(12)(2)(125)(47)(47
8)(8)(258)(58)(25)(1245678)(125
68)(125568)(1256)(12569)(256
8)(25568)(256)(2569)(568)(356
8)(5)(56)(6)(356)(569)(35669)(3
6)(3569)(69)(369)(3)(7)(9)
 \\   
 & = 
(4)(14)(124)(2)(1)(125)(47)(47
8)(8)(258)(58)(25)(1245678)(125
68)(125568)(1256)(12569)(256
8)(25568)(256)(2569)(568)(356
8)(5)(56)(6)(356)(569)(35669)(3
6)(3569)(69)(369)(3)(7)(9)
 \\   
 & = 
(4)(14)(124)(2)(1)(125)(47)(47
8)(8)(258)(25)(58)(1245678)(125
68)(125568)(1256)(12569)(256
8)(25568)(256)(2569)(568)(356
8)(5)(56)(6)(356)(569)(35669)(3
6)(3569)(69)(369)(3)(7)(9)
 \\   
 & = 
(4)(14)(124)(2)(1)(125)(47)(478)(2
5)(8)(58)(1245678)(12568)(125
568)(1256)(12569)(2568)(2556
8)(256)(2569)(568)(3568)(5)(5
6)(6)(356)(569)(35669)(36)(356
9)(69)(369)(3)(7)(9)
 \\   
 & = 
(4)(14)(124)(2)(1)(125)(47)(25)(47
8)(8)(58)(1245678)(12568)(125
568)(1256)(12569)(2568)(2556
8)(256)(2569)(568)(3568)(5)(5
6)(6)(356)(569)(35669)(36)(356
9)(69)(369)(3)(7)(9)
 \\   
 & = 
(4)(14)(124)(2)(1)(125)(25)(245
7)(47)(478)(8)(58)(1245678)(125
68)(125568)(1256)(12569)(256
8)(25568)(256)(2569)(568)(356
8)(5)(56)(6)(356)(569)(35669)(3
6)(3569)(69)(369)(3)(7)(9)
 \\   
 & = 
(4)(14)(124)(2)(25)(1)(2457)(47)(4
78)(8)(58)(1245678)(12568)(12
5568)(1256)(12569)(2568)(255
68)(256)(2569)(568)(3568)(5)(5
6)(6)(356)(569)(35669)(36)(356
9)(69)(369)(3)(7)(9)
 \\   
 & = 
(4)(14)(124)(2)(25)(2457)(1)(47)(4
78)(8)(58)(1245678)(12568)(12
5568)(1256)(12569)(2568)(255
68)(256)(2569)(568)(3568)(5)(5
6)(6)(356)(569)(35669)(36)(356
9)(69)(369)(3)(7)(9)
 \\   
 & = 
(4)(14)(124)(2)(25)(2457)(47)(14
7)(1)(478)(8)(58)(1245678)(1256
8)(125568)(1256)(12569)(256
8)(25568)(256)(2569)(568)(356
8)(5)(56)(6)(356)(569)(35669)(3
6)(3569)(69)(369)(3)(7)(9)
 \\   
 & = 
(4)(14)(124)(2)(25)(2457)(47)(14
7)(478)(1478)(1)(8)(58)(124567
8)(12568)(125568)(1256)(1256
9)(2568)(25568)(256)(2569)(56
8)(3568)(5)(56)(6)(356)(569)(356
69)(36)(3569)(69)(369)(3)(7)(9)
 \\   
 & = 
(4)(14)(124)(2)(25)(2457)(47)(14
7)(478)(1478)(8)(1)(58)(124567
8)(12568)(125568)(1256)(1256
9)(2568)(25568)(256)(2569)(56
8)(3568)(5)(56)(6)(356)(569)(356
69)(36)(3569)(69)(369)(3)(7)(9)
 \\   
 & = 
(4)(14)(124)(2)(25)(2457)(47)(14
7)(478)(1478)(8)(58)(1)(124567
8)(12568)(125568)(1256)(1256
9)(2568)(25568)(256)(2569)(56
8)(3568)(5)(56)(6)(356)(569)(356
69)(36)(3569)(69)(369)(3)(7)(9)
 \\   
 & = 
(4)(14)(124)(2)(25)(2457)(47)(14
7)(478)(1478)(8)(58)(124567
8)(1)(12568)(125568)(1256)(12
569)(2568)(25568)(256)(256
9)(568)(3568)(5)(56)(6)(356)(56
9)(35669)(36)(3569)(69)(36
9)(3)(7)(9)
 \\   
 & = 
(4)(14)(124)(2)(25)(2457)(47)(14
7)(478)(1478)(8)(58)(124567
8)(1)(12568)(125568)(1256)(25
68)(12569)(25568)(256)(256
9)(568)(3568)(5)(56)(6)(356)(56
9)(35669)(36)(3569)(69)(36
9)(3)(7)(9)
 \\   
 & = 
(4)(14)(124)(2)(25)(2457)(47)(14
7)(478)(1478)(8)(58)(124567
8)(1)(12568)(125568)(2568)(12
56)(12569)(25568)(256)(256
9)(568)(3568)(5)(56)(6)(356)(56
9)(35669)(36)(3569)(69)(36
9)(3)(7)(9)
 \\   
 & = 
(4)(14)(124)(2)(25)(2457)(47)(14
7)(478)(1478)(8)(58)(124567
8)(1)(12568)(2568)(125568)(12
56)(12569)(25568)(256)(256
9)(568)(3568)(5)(56)(6)(356)(56
9)(35669)(36)(3569)(69)(36
9)(3)(7)(9)
 \\  
 & = 
(4)(14)(124)(2)(25)(2457)(47)(14
7)(478)(1478)(8)(58)(1245678)(2
568)(1)(125568)(1256)(1256
9)(25568)(256)(2569)(568)(356
8)(5)(56)(6)(356)(569)(35669)(3
6)(3569)(69)(369)(3)(7)(9)
 \\   
 & = 
(4)(14)(124)(2)(25)(2457)(47)(14
7)(478)(1478)(8)(58)(1245678)(2
568)(1)(125568)(1256)(2556
8)(12569)(256)(2569)(568)(356
8)(5)(56)(6)(356)(569)(35669)(3
6)(3569)(69)(369)(3)(7)(9)
 \\   
 & = 
(4)(14)(124)(2)(25)(2457)(47)(14
7)(478)(1478)(8)(58)(1245678)(2
568)(1)(125568)(25568)(125
6)(12569)(256)(2569)(568)(356
8)(5)(56)(6)(356)(569)(35669)(3
6)(3569)(69)(369)(3)(7)(9)
 \\  
 & = 
(4)(14)(124)(2)(25)(2457)(47)(14
7)(478)(1478)(8)(58)(1245678)(2
568)(25568)(1)(1256)(12569)(2
56)(2569)(568)(3568)(5)(56)(6)(3
56)(569)(35669)(36)(3569)(69)(3
69)(3)(7)(9)
 \\   
 & = 
(4)(14)(124)(2)(25)(2457)(47)(14
7)(478)(1478)(8)(58)(1245678)(2
568)(25568)(1)(1256)(256)(125
69)(2569)(568)(3568)(5)(56)(6)(3
56)(569)(35669)(36)(3569)(69)(3
69)(3)(7)(9)
 \\   
 & = 
(4)(14)(124)(2)(25)(2457)(47)(14
7)(478)(1478)(8)(58)(1245678)(2
568)(25568)(256)(1)(12569)(25
69)(568)(3568)(5)(56)(6)(356)(56
9)(35669)(36)(3569)(69)(36
9)(3)(7)(9)
 \\   
 & = 
(4)(14)(124)(2)(25)(2457)(47)(14
7)(478)(1478)(8)(58)(1245678)(2
568)(25568)(256)(2569)(1)(56
8)(3568)(5)(56)(6)(356)(569)(356
69)(36)(3569)(69)(369)(3)(7)(9)
 \\   
 & = 
(4)(14)(124)(2)(25)(2457)(47)(14
7)(478)(1478)(8)(58)(1245678)(2
568)(25568)(256)(2569)(56
8)(1)(3568)(5)(56)(6)(356)(569)(3
5669)(36)(3569)(69)(369)(3)(7)(9)
 \\   
 & = 
(4)(14)(124)(2)(25)(2457)(47)(14
7)(478)(1478)(8)(58)(1245678)(2
568)(25568)(256)(2569)(568)(3
568)(1)(5)(56)(6)(356)(569)(3566
9)(36)(3569)(69)(369)(3)(7)(9)
 \\   
 & = 
(4)(14)(124)(2)(25)(2457)(47)(14
7)(478)(1478)(8)(58)(1245678)(2
568)(25568)(256)(2569)(568)(3
568)(1)(6)(5)(356)(569)(35669)(3
6)(3569)(69)(369)(3)(7)(9)
 \\   
 & = 
(4)(14)(124)(2)(25)(2457)(47)(14
7)(478)(1478)(8)(58)(1245678)(2
568)(25568)(256)(2569)(568)(3
568)(1)(6)(5)(356)(36)(569)(356
9)(69)(369)(3)(7)(9)
 \\   
 & = 
(4)(14)(124)(2)(25)(2457)(47)(14
7)(478)(1478)(8)(58)(1245678)(2
568)(25568)(256)(2569)(568)(3
568)(1)(6)(36)(5)(569)(3569)(6
9)(369)(3)(7)(9)
 \\  
 & = 
(4)(14)(124)(2)(25)(2457)(47)(14
7)(478)(1478)(8)(58)(1245678)(2
568)(25568)(256)(2569)(568)(3
568)(1)(6)(36)(5)(569)(69)(356
9)(369)(3)(7)(9)
 \\   
 & = 
(4)(14)(124)(2)(25)(2457)(47)(14
7)(478)(1478)(8)(58)(1245678)(2
568)(25568)(256)(2569)(568)(3
568)(1)(6)(36)(69)(5)(3569)(36
9)(3)(7)(9)
 \\   
 & = 
(4)(14)(124)(2)(25)(2457)(47)(14
7)(478)(1478)(8)(58)(1245678)(2
568)(25568)(256)(2569)(568)(3
568)(1)(6)(36)(69)(369)(5)(3)(7)(9)
 \\   
 & = 
(4)(14)(124)(2)(25)(2457)(47)(14
7)(478)(1478)(8)(58)(1245678)(2
568)(25568)(256)(2569)(568)(3
568)(1)(6)(36)(69)(369)(3)(5)(7)(9)
 \\   
 & = 
(4)(2)(14)(25)(2457)(47)(147)(47
8)(1478)(8)(58)(1245678)(256
8)(25568)(256)(2569)(568)(356
8)(1)(6)(36)(69)(369)(3)(5)(7)(9)
 \\   
 & = 
(2)(4)(14)(25)(2457)(47)(147)(47
8)(1478)(8)(58)(1245678)(256
8)(25568)(256)(2569)(568)(356
8)(1)(6)(36)(69)(369)(3)(5)(7)(9)
 \\   
 & = 
(2)(4)(14)(25)(2457)(47)(147)(47
8)(1478)(8)(58)(1245678)(256
8)(25568)(256)(2569)(568)(356
8)(6)(1)(36)(69)(369)(3)(5)(7)(9)
 \\   
 & = 
(2)(4)(14)(25)(2457)(47)(147)(47
8)(1478)(8)(58)(1245678)(256
8)(25568)(256)(2569)(568)(356
8)(6)(36)(1)(69)(369)(3)(5)(7)(9)
 \\  
 & = 
(2)(4)(14)(25)(2457)(47)(147)(47
8)(1478)(8)(58)(1245678)(256
8)(25568)(256)(2569)(568)(356
8)(6)(36)(69)(1)(369)(3)(5)(7)(9)
 \\   
 & = 
(2)(4)(14)(25)(2457)(47)(147)(47
8)(1478)(8)(58)(1245678)(256
8)(25568)(256)(2569)(568)(356
8)(6)(36)(69)(369)(1)(3)(5)(7)(9)
 \\   
 & = 
(2)(4)(14)(25)(2457)(47)(147)(47
8)(1478)(8)(58)(1245678)(256
8)(25568)(256)(2569)(568)(6)(35
68)(36)(69)(369)(1)(3)(5)(7)(9)
 \\   
 & = 
(2)(4)(14)(25)(2457)(47)(147)(47
8)(1478)(8)(1245678)(58)(256
8)(25568)(256)(2569)(568)(6)(35
68)(36)(69)(369)(1)(3)(5)(7)(9)
 \\   
 & = 
(2)(4)(14)(25)(2457)(47)(147)(47
8)(1478)(8)(1245678)(2568)(5
8)(25568)(256)(2569)(568)(6)(35
68)(36)(69)(369)(1)(3)(5)(7)(9)
 \\   
 & = 
(2)(4)(14)(25)(2457)(47)(147)(47
8)(1478)(8)(1245678)(2568)(25
6)(58)(2569)(568)(6)(3568)(36)(6
9)(369)(1)(3)(5)(7)(9)
 \\   
 & = 
(2)(4)(14)(25)(2457)(47)(147)(47
8)(1478)(8)(1245678)(2568)(25
6)(2569)(58)(568)(6)(3568)(36)(6
9)(369)(1)(3)(5)(7)(9)
 \\   
 & = 
(2)(4)(14)(25)(2457)(47)(147)(47
8)(1478)(8)(1245678)(2568)(25
6)(2569)(6)(58)(3568)(36)(69)(36
9)(1)(3)(5)(7)(9)
 \\   
 & = 
(2)(4)(14)(25)(2457)(47)(147)(47
8)(1478)(8)(1245678)(2568)(25
6)(6)(2569)(58)(3568)(36)(69)(36
9)(1)(3)(5)(7)(9)
 \\   
 & = 
(2)(4)(14)(47)(25)(147)(478)(147
8)(8)(1245678)(2568)(256)(6)(25
69)(58)(3568)(36)(69)(36
9)(1)(3)(5)(7)(9)
 \\   
 & = 
(2)(4)(14)(47)(147)(25)(478)(147
8)(8)(1245678)(2568)(256)(6)(25
69)(58)(3568)(36)(69)(36
9)(1)(3)(5)(7)(9)
 \\   
 & = 
(2)(4)(14)(47)(147)(478)(25)(1478)(8)(1245678)(2568)(256)(6)(25
69)(58)(3568)(36)(69)(36
9)(1)(3)(5)(7)(9)
 \\   
 & = 
(2)(4)(14)(47)(147)(478)(1478)(124578)(25)(8)(1245678)(2568)(256)(6)(2569)(58)(3568)(36)(69)(369)(1)(3)(5)(7)(9)
 \\   
 & = 
(2)(4)(14)(47)(147)(478)(1478)(124578)(8)(258)(25)(1245678)(2568)(256)(6)(2569)(58)(3568)(36)(69)(369)(1)(3)(5)(7)(9)
 \\   
 & = 
(2)(4)(14)(47)(147)(478)(1478)(124578)(8)(258)(1245678)(25)(2568)(256)(6)(2569)(58)(3568)(36)(69)(369)(1)(3)(5)(7)(9)
 \\   
 & = 
(2)(4)(14)(47)(147)(478)(1478)(124578)(8)(258)(1245678)(2568)(25)(256)(6)(2569)(58)(3568)(36)(69)(369)(1)(3)(5)(7)(9)
 \\   
 & = 
(2)(4)(14)(47)(147)(478)(1478)(124578)(8)(258)(1245678)(2568)(6)(25)(2569)(58)(3568)(36)(69)(369)(1)(3)(5)(7)(9)
 \\   
 & = 
(2)(4)(14)(47)(147)(478)(1478)(124578)(8)(1245678)(258)(2568)(6)(25)(2569)(58)(3568)(36)(69)(369)(1)(3)(5)(7)(9)
 \\   
 & = 
(2)(4)(14)(47)(147)(478)(1478)(124578)(8)(1245678)(6)(258)(25)(2569)(58)(3568)(36)(69)(369)(1)(3)(5)(7)(9)
 \\   
 & = 
(2)(4)(14)(47)(147)(478)(1478)(8)(124578)(1245678)(6)(258)(25)(2569)(58)(3568)(36)(69)(369)(1)(3)(5)(7)(9)
 \\   
 & = 
(2)(4)(14)(47)(147)(478)(1478)(8)(6)(124578)(258)(25)(2569)(58)(3568)(36)(69)(369)(1)(3)(5)(7)(9)
 \\   
 & = 
(2)(4)(14)(47)(147)(478)(1478)(6)(8)(124578)(258)(25)(2569)(58)(3568)(36)(69)(369)(1)(3)(5)(7)(9)
 \\   
 & = 
(2)(4)(14)(47)(147)(478)(6)(1478)(8)(124578)(258)(25)(2569)(58)(3568)(36)(69)(369)(1)(3)(5)(7)(9)
 \\   
 & = 
(2)(4)(14)(47)(147)(6)(478)(1478)(8)(124578)(258)(25)(2569)(58)(3568)(36)(69)(369)(1)(3)(5)(7)(9)
 \\   
 & = 
(2)(4)(14)(47)(6)(147)(478)(1478)(8)(124578)(258)(25)(2569)(58)(3568)(36)(69)(369)(1)(3)(5)(7)(9)
 \\   
 & = 
(2)(4)(14)(6)(47)(147)(478)(1478)(8)(124578)(258)(25)(2569)(58)(3568)(36)(69)(369)(1)(3)(5)(7)(9)
 \\   
 & = 
(2)(4)(6)(14)(47)(147)(478)(1478)(8)(124578)(258)(25)(2569)(58)(3568)(36)(69)(369)(1)(3)(5)(7)(9)
 \\   
 & = 
(2)(4)(6)(14)(47)(478)(147)(1478)(8)(124578)(258)(25)(2569)(58)(3568)(36)(69)(369)(1)(3)(5)(7)(9)
 \\   
 & = 
(2)(4)(6)(14)(47)(478)(8)(147)(124578)(258)(25)(2569)(58)(3568)(36)(69)(369)(1)(3)(5)(7)(9)
 \\   
 & = 
(2)(4)(6)(14)(8)(47)(147)(124578)(258)(25)(2569)(58)(3568)(36)(69)(369)(1)(3)(5)(7)(9)
 \\   
 & = 
(2)(4)(6)(8)(14)(47)(147)(124578)(258)(25)(2569)(58)(3568)(36)(69)(369)(1)(3)(5)(7)(9)
 \\   
 & = 
(2)(4)(6)(8)(14)(47)(147)(124578)(258)(25)(2569)(36)(58)(69)(369)(1)(3)(5)(7)(9)
 \\   
 & = 
(2)(4)(6)(8)(14)(47)(147)(124578)(258)(25)(36)(2569)(58)(69)(369)(1)(3)(5)(7)(9)
 \\   
 & = 
(2)(4)(6)(8)(14)(47)(147)(124578)(258)(36)(25)(2569)(58)(69)(369)(1)(3)(5)(7)(9)
 \\   
 & = 
(2)(4)(6)(8)(14)(47)(147)(124578)(36)(258)(25)(2569)(58)(69)(369)(1)(3)(5)(7)(9)
 \\   
 & = 
(2)(4)(6)(8)(14)(47)(147)(36)(124578)(258)(25)(2569)(58)(69)(369)(1)(3)(5)(7)(9)
 \\   
 & = 
(2)(4)(6)(8)(14)(47)(36)(147)(124578)(258)(25)(2569)(58)(69)(369)(1)(3)(5)(7)(9)
 \\   
 & = 
(2)(4)(6)(8)(14)(36)(47)(147)(124578)(258)(25)(2569)(58)(69)(369)(1)(3)(5)(7)(9)
 \\   
 & = 
(2)(4)(6)(8)(14)(36)(47)(258)(147)(25)(2569)(58)(69)(369)(1)(3)(5)(7)(9)
 \\   
 & = 
(2)(4)(6)(8)(14)(36)(258)(47)(147)(25)(2569)(58)(69)(369)(1)(3)(5)(7)(9)
 \\   
 & = 
(2)(4)(6)(8)(14)(36)(258)(47)(147)(25)(2569)(69)(58)(369)(1)(3)(5)(7)(9)
 \\   
 & = 
(2)(4)(6)(8)(14)(36)(258)(47)(147)(69)(25)(58)(369)(1)(3)(5)(7)(9)
 \\   
 & = 
(2)(4)(6)(8)(14)(36)(258)(47)(69)(147)(25)(58)(369)(1)(3)(5)(7)(9)
 \\   
 & = 
(2)(4)(6)(8)(14)(36)(258)(47)(69)(147)(58)(25)(369)(1)(3)(5)(7)(9)
 \\   
 & = 
(2)(4)(6)(8)(14)(36)(258)(47)(69)(58)(147)(25)(369)(1)(3)(5)(7)(9)
 \\   
 & = 
(2)(4)(6)(8)(14)(36)(258)(47)(69)(58)(147)(369)(25)(1)(3)(5)(7)(9)
\end{align*}
}%